\documentclass[preprint,noinfoline,12pt]{imsart}
\usepackage{amsmath,amssymb,amsthm,enumerate,framed,graphicx,array,color,multirow,mathrsfs}

\usepackage[utf8]{inputenc}

\usepackage{a4wide}

\allowdisplaybreaks[4]
\numberwithin{equation}{section}

\usepackage[usenames,dvipsnames,svgnames,table]{xcolor}

\usepackage{mathrsfs}
\usepackage{euscript}

\theoremstyle{plain}
\newtheorem{theorem}{Theorem}[section]
\newtheorem{lemma}{Lemma}[section]
\newtheorem{corollary}{Corollary}[section]
\newtheorem{proposition}{Proposition}[section]

\newtheorem{remark}{Remark}[section]

\def\R{{\mathbb R}}

\def\E{\mathbb E}
\def\L{\mathcal L}

\usepackage[authoryear]{natbib}
\begin{document}


%
\begin{frontmatter}

  \title{High-dimensional  central limit theorems for eigenvalue  distributions of generalized Wishart
    processes}
  
  \runtitle{High-dimensional CLT for general Wishart processes}

  \begin{aug}
    \author{\fnms{~ Jian} \snm{Song}\ead[label=e1]{txjsong@hotmail.com}}
    \and
    \author{\fnms{~ Jianfeng} \snm{Yao~}\ead[label=e2]{jeffyao@hku.hk}}
    \and
    \author{\fnms{~ Wangjun} \snm{Yuan}\ead[label=e3]{ywangjun@connect.hku.hk}}
    
    \affiliation{Shandong University and  The University of Hong Kong}
    \runauthor{J. Song,   J. Yao \& W. Yuan}

    \address{School of Mathematics, Shandong University\\
      \printead{e1}}
    
    \address{ 
      Department of Statistics and Actuarial Science, 
      The University of Hong Kong\\
      \printead{e2}
    }

    \address{
      Department of Mathematics, 
      The University of Hong Kong\\
      \printead{e3}
    }
  \end{aug}

  \begin{abstract}

We consider eigenvalues of generalized Wishart processes as well as particle systems, of which the empirical measures converge to deterministic measures as the dimension goes to infinity. In this paper, we obtain central limit theorems to characterize the fluctuations of the empirical measures around the limit measures by using stochastic calculus. As applications, central limit theorems for the Dyson's Brownian motion and the eigenvalues of the Wishart process are recovered under slightly more general initial conditions, and a central limit theorem for the eigenvalues of a symmetric Ornstein-Uhlenbeck matrix process is obtained. 
    
  \end{abstract}
  


\end{frontmatter}

\section{Introduction}\label{sec:intro}

Recently 
general stochastic differential equations (SDEs)
on the group of symmetric matrices have attracted much interest.
A prominent example is 
the following generalized Wishart process  introduced  in \cite{Graczyk2013},
\begin{align} \label{matrix SDE}
  d X_t^N = g_N(X_t^N) dB_t h_N(X_t^N) + h_N(X_t^N) dB_t^\intercal g_N(X_t^N) + b_N(X_t^N) dt, \quad t \ge 0.
\end{align}
Here, $B_t$ is a Brownian matrix of dimension $N \times N$, and the
continuous functions $g_N, h_N, b_N: \mathbb{R} \rightarrow
\mathbb{R}$ act on the spectrum of $X_t^N$ (a function $f$ acts on the spectrum of a symmetric matrix $X = \sum_{j=1}^N \alpha_j u_ju_j^\intercal$ with eigenvalues $(\alpha_j)$ and eigenvectors $(u_j)$ if  $f(X)=\sum_{j=1}^N f( \alpha_j) u_ju_j^\intercal$). 
The generalized Wishart process \eqref{matrix SDE} includes as simple
examples the following well-known matrix-valued stochastic processes:  
the celebrated symmetric Brownian motion \citep{Dyson62}, 
the Wishart process \citep{Bru91}, and the symmetric matrix process  whose entries are independent Ornstein-Uhlenbeck processes
\citep{Chan1992}.

Suppose that $\lambda_1^N(t) \le \lambda_2^N(t) \le \ldots \le \lambda_N^N(t)$ are the eigenvalues of $X_t^N$. According to Theorem 3 in \citet{Graczyk2013}, if $\lambda_1^N(0) < \lambda_2^N(0) < \cdots < \lambda_N^N(0)$, then before the first collision time
\begin{align*}
	\tau_N = \inf\{t>0: \exists \ i \neq j, ~\lambda_i(t) = \lambda_j(t) \},
\end{align*}
the eigenvalues satisfy the following system of SDEs: for $1\le i \le N$,
\begin{align} \label{eigenvalue SDE}
	d \lambda_i^N(t)
	= 2 g_N(\lambda_i^N(t)) h_N(\lambda_i^N(t)) dW_i(t) + \left( b_N(\lambda_i^N(t)) + \sum_{j:j \neq i} \dfrac{G_N(\lambda_i^N(t),\lambda_j^N(t))}{\lambda_i^N(t) - \lambda_j^N(t)} \right) dt,
\end{align}
where $\{W_i, 1 \le i \le N\}$ are independent Brownian motions and 
\begin{equation}\label{eq-Gn}
G_N(x,y) = g_N^2(x) h_N^2(y) + g_N^2(y) h_N^2(x).
\end{equation}
In \citet{Graczyk2013,Graczyk2014}, some other conditions on the coefficient functions were imposed to ensure that \eqref{eigenvalue SDE}  has a unique strong solution and the collision time $\tau_N$ is infinite almost surely.

Let $L_N(t)$ be the empirical measure of the eigenvalues $\{\lambda_i^N(t), 1\le i\le N\}$, i.e.,
\begin{align}\label{eq-empi-meas}
  L_N(t) = \dfrac{1}{N} \sum_{i=1}^N \delta_{\lambda_i^N(t)}.
\end{align}
In connection with  the theory of random matrices, it is of interest to investigate possible limits of these empirical measures $\{ L_N(t)\}$ when $N$ grows to infinity ({\em high-dimensional limits}). The literature on such high-dimensional limits is sparse. An early result is the derivation of the Wigner semi-circle law as the only equilibrium point (with finite moments of all orders) of  the  equation satisfied by the limit of eigenvalue empirical measure process in \cite{Chan1992}, where the symmetric matrix process has independent Ornstein-Uhlenbeck processes as its entries. The results were later  generalized in \cite{Rogers1993} to the following interacting particle system
\begin{align*}
  dX_i = \sqrt{\dfrac{2\alpha}{N}} dB_i + \left( - \theta X_i + \dfrac{\alpha}{N} \sum_{j:j\neq i} \dfrac{1}{X_i - X_j} \right) dt, \quad 1 \le i \le N, t \ge 0.
\end{align*}
\cite{Cepa1997} further generalized  these SDEs to 
\begin{align*}
  dX_i = \sigma(X_i) dB_i + \left( b(X_i) + \sum_{j:j\neq i} \dfrac{\gamma}{X_i - X_j} \right) dt, \quad 1 \le i \le N, t \ge 0,
\end{align*}
with some coefficient functions $b$, $\sigma$ and constant $\gamma$.
Another important case is the Mar\v{c}enko-Pastur law for 
the eigenvalue empirical measure process derived in \cite{Duvillard2001}.

The eigenvalue SDEs \eqref{eigenvalue SDE} generalize the eigenvalue SDEs in \cite{Chan1992} and \cite{Duvillard2001}, as well as the particle system in \cite{Rogers1993}. High-dimensional limits for these
eigenvalue SDEs appeared very recently in \cite{Song2019} and \cite{Malecki-arxiv}. Particularly in the former article, it was proved that under  proper conditions, $\{L_N(t), t \in [0,T]\}_{N \in \mathbb{N}}$ is relatively compact in $(C[0, T], M_1(\R))$ almost surely. Here $M_1(\R)$ is the set of probability measures on $\R$ endowed with the topology induced by the weak convergence of measures.  Furthermore, any limit measure $\{\mu_t, t \in [0,T]\}$ from a converging subsequence satisfies
\begin{align} \label{limit measure equation Stieltjes}
	\int \dfrac{\mu_t(dx)}{z-x}
	=& \int \dfrac{\mu_0(dx)}{z-x} + \int_{0}^t \left[ \int \dfrac{b(x)}{(z-x)^2} \mu_s(dx) \right] ds\notag \\
	&+ \int_{0}^t \left[ \iint \dfrac{G(x,y)}{(z-x) (z-y)^2} \mu_s(dx) \mu_s(dy) \right] ds, ~~ \forall z \in \mathbb{C} \setminus \mathbb{R},
\end{align}
with
\begin{equation}\label{eq-bg}
b(x)=\lim_{N\to\infty}b_N(x)~~ \text{ and  } ~~ G(x,y)=\lim_{N\to \infty} NG_N(x,y),
\end{equation}
 uniformly. Note that \cite{Song2019} provided examples where such limit $\{\mu_t,
t \in [0,T]\}$ is unique. However, conditions for the uniqueness are
still unknown for the general system \eqref{limit measure equation Stieltjes}. 

In this paper, we study the fluctuations of $\{L_N(t), t \in [0,T]\}$ around  the limit $\{\mu_t, t \in [0,T]\}$. Up to considering a subsequence, the theory is here developed, without loss of generality, 
by assuming the convergence of the whole sequence $\{L_N(t), t \in [0,T]\}$ to $\{\mu_t, t \in [0,T]\}$. Consider the random fluctuations 
\begin{align}\label{eq-functional}
	\mathcal{L}_t^N(f) = N \langle f, L_N(t) - \mu_t \rangle
	= \sum_{i=1}^N f(\lambda_i^N(t)) - N \langle f, \mu_t \rangle,
\end{align}
for $f \in \mathbb{F}$,  where $\mathbb F$ is an appropriate space of test functions given by \eqref{space F} or \eqref{space F'} in Section \ref{section-CLT}. The main purpose of the paper is to find a Gaussian limit for the centered process
\begin{align}\label{eq-Q}
	Q_t^N(f) &= \mathcal{L}_t^N(f) - \mathcal{L}_0^N(f) - \int_{0}^t \mathcal{L}_s^N(f'b) ds - \dfrac{1}{2} \int_{0}^t \langle f''(x) G(x,x), \mu_s \rangle ds\notag \\
	&\quad- \int_0^t \mathcal{L}_s^N\left( \int \dfrac{f'(x) - f'(y)}{x - y} G(x,y) \mu_s(dx) \right) ds\notag \\
	&\quad- \dfrac{N}{2} \int_{0}^t \iint \dfrac{f'(x) - f'(y)}{x - y} G(x,y) [L_N(s)(dx) - \mu_s(dx)] [L_N(s)(dy) - \mu_s(dy)] ds,
\end{align}
as $N$ goes to infinity. To our best knowledge, the literature on this topic is quite limited, and we only refer to \cite{cd01, Anderson2010} which concern the cases of Dyson's Brownian motion and Wishart process.

Now, we briefly explain the structure of this paper as follows.

The main results in this paper are presented in Section \ref{section-CLT}. The central limit theorem (CLT) for the empirical measure of the eigenvalues \eqref{eigenvalue SDE} is obtained in Section \ref{sec:wishart}. The same techniques allow to establish the CLT   in Scetion \ref{sec:particle} for the empirical measure of a class of particle system \eqref{SDE-particle}  which was introduced in \cite{Graczyk2014} 
as an generalization of \eqref{eigenvalue SDE}.  Note that in particular \eqref{SDE-particle} includes  the particle system studied in \cite{Cepa1997} as a special example. 
 
In Section \ref{section-application}, we apply the results  in Section \ref{section-CLT} to obtain the CLTs for the eigenvalues of Wishart process in Section \ref{sec-wishart}, for the Dyson's Brownian motion in Section \ref{sec-dyson}, and for the eigenvalues of symmetric Ornstein-Uhlenbeck matrix process in Section \ref{sec-ou}, respectively.  Note that for these three cases, under proper initial conditions, we can obtain the boundedness for the eigenvalues/particles, which enables us to obtain more precise CLTs for a wider class of test functions. In order to obtain such bounds starting from more general initial conditions, inspired by \cite{sniady02} and \cite{Anderson2010},  in Section \ref{section-comparison} we develop a comparison principle for SDE \eqref{eigenvalue SDE} and particle system \eqref{SDE-particle}.  This comparison principle also allows to extend the CLTs developed in Section \ref{section-application} to a wider class of particles systems (Corollaries \ref{general wishart}, \ref{general Dyson} and \ref{general OU}).

Furthermore, due to the special structures of the Wishart process, the Dyson's Brownian motion, and the Ornstein-Uhlenbeck matrix process, we are able to directly characterize the fluctuations $\{\L_t(x^n), t\in[0,T]\}_{n\in\mathbb N}$, where $\L_t(x^n)$ is the limit of $\L_t^N(x^n)$, by recursive formulas (See Theorems \ref{CLT for Wishart}, \ref{CLT for Dyson}, \ref{CLT for OU} and the remarks thereafter).
 For the Dyson's Brownian motion, the CLT was obtained in \cite{cd01} with null initial condition, and the restriction on the initial condition was later relaxed in \cite{Anderson2010}. This CLT is recovered in Section \ref{sec-dyson} with slightly more general initial condition. For the eigenvalue processes of Wishart process, the CLT was obtained in \cite{cd01} again with null initial condition, and it is now extended in Section \ref{sec-wishart} allowing more general initial conditions. Lastly, the CLT obtained in Section \ref{sec-ou} for the eigenvalue process of Ornstein-Uhlenbeck matrix process seems new.
 
Finally, in Section \ref{sec-appendix} some useful lemmas are provided.

\section{Central limit theorems}\label{section-CLT}
In this section, we prove our main results of the CLTs for eigenvalues of general Wishart processes in Section \ref{sec:wishart} and for particle systems in Section \ref{sec:particle}, repsectively.
\subsection{Central limit theorem for eigenvalues of general Wishart processes}\label{sec:wishart}

In this subsection, we study the CLT for the empirical measure \eqref{eq-empi-meas} of the eigenvalues  \eqref{eigenvalue SDE} of generalized Wishart process \eqref{matrix SDE}. 

Recall that the functions $b(x)$ and $G(x,x)$ are defined in \eqref{eq-bg}, and $Q_t^N(f)$ is defined in \eqref{eq-Q}. We use the following space of test functions
\begin{align}
	\mathbb{F} = &\Bigg\{ f \in C_b^2(\mathbb{R}): ~~\| f'(x) b(x) \|_{L^{\infty}(\mathbb{R})} < \infty,\nonumber \\
	& \qquad \left\| \dfrac{f'(x) - f'(y)}{x-y} G(x,y) \right\|_{L^{\infty}(\mathbb{R}^2)} < \infty, ~~\| (f'(x))^2 G(x,x) \|_{L^{\infty}(\mathbb{R})} < \infty \Bigg\}.\label{space F}
\end{align}

\begin{theorem} \label{Thm-Wishart}
Assume that the limit functions $b(x)$ and $G(x,y)$ are continuous and satisfy
\begin{align} \label{conv-speed}
\begin{aligned}
	\lim_{N \rightarrow \infty} N \|b_N(x) - b(x)\|_{L^{\infty}(\mathbb{R})} = 0, \\
	\lim_{N \rightarrow \infty} N \|NG_N(x,y) - G(x,y)\|_{L^{\infty}(\mathbb{R}^2)} = 0.
\end{aligned}
\end{align}
Also assume that \eqref{eigenvalue SDE} has a non-exploding and non-colliding strong solution, such that the sequence of the empirical measures  $\{L_N(t), t \in [0,T]\}_{N \in \mathbb{N}}$ given by \eqref{eq-empi-meas} converges weakly to $\{\mu_t, t\in [0,T]\}$. 

Then, for any $k \in \mathbb{N}$ and any $f_1, \ldots, f_k \in \mathbb{F}$, as $N$ goes to infinity, $(Q_t^N(f_1), \ldots, Q_t^N(f_k))_{t \in [0,T]}$ converges in distribution to a Gaussian process $(G_t(f_1), \ldots, G_t(f_k))_{t \in [0,T]}$ with mean  zero and covariance
\begin{align} \label{cov}
	\mathbb{E} \left[ G_t(f_i) G_s(f_j) \right]
	= 2 \int_0^{t \wedge s} \langle f_i'(x) f_j'(x) G(x,x), \mu_u \rangle du, ~~ 1 \le i,j \le k.
\end{align}
\end{theorem}


\begin{proof} 

By It\^o's formula (see \cite{Song2019} for more details), for $f \in C^2[0,T]$,
\begin{align}
  \nonumber
  \langle f, L_N(t) \rangle
  &= \langle f, L_N(0) \rangle + M_f^N(t) + \int_{0}^t \langle f'b_N, L_N(s) \rangle ds + \int_{0}^t \langle f''g_N^2h_N^2, L_N(s) \rangle ds  \\
	&\quad + \dfrac{N}{2} \int_{0}^t \iint \dfrac{f'(x) - f'(y)}{x - y} G_N (x, y) L_N(s)(dx) L_N(s)(dy) ds,\label{pair formula for empirical measure}
\end{align}
 where we use the convention
$\frac{f'(x) - f'(y)}{x - y} = f''(x)$
on $\{(x,y) \in \mathbb{R}^2: x = y\}$, and $M_f^N(t)$ is a local martingale,
\begin{align} \label{martingale term}
	M_f^N(t) = \dfrac{2}{N} \sum_{i=1}^N \int_{0}^t f'(\lambda_i^N(s)) g_N(\lambda_i^N(s)) h_N(\lambda_i^N(s)) dW_i(s),
\end{align}
with quadratic variation
\begin{align} \label{quadratic variation of martingale}
	\langle M_f^N \rangle_t
	= \dfrac{4}{N} \int_{0}^t \langle |f' g_N h_N|^2, L_N(s) \rangle ds
	= \dfrac{2}{N} \int_{0}^t \langle |f'(x)|^2 G_N(x,x), L_N(s) \rangle ds.
\end{align} 
On the other hand, for $f \in \mathbb{F}$, under the condition \eqref{conv-speed},  one may apply the approach used in the proof of Theorem 2.2 in \cite{Song2019} to get
\begin{align} \label{limit measure equation general}
	\langle f, \mu_t \rangle
	&= \langle f, \mu_0 \rangle + \int_{0}^t \langle f'b, \mu_s \rangle ds + \dfrac{1}{2} \int_{0}^t \iint \dfrac{f'(x) - f'(y)}{x - y} G(x,y) \mu_s(dx) \mu_s(dy) ds.
\end{align}
(Indeed, the proof of Theorem 2.2 in \cite{Song2019} deals with the special case  $f(x) = (z-x)^{-1}$ with $z\in \mathbb C\backslash \R$.) 

Thus, \eqref{pair formula for empirical measure} and  \eqref{limit measure equation general} yield 
\begin{align} \label{measure error}
	&\quad \L_t^N(f)=N \langle f, L_N(t) - \mu_t \rangle  \\
	&= N \langle f, L_N(0) - \mu_0 \rangle + N M_f^N(t) \nonumber \\
	&\quad + N \int_{0}^t \langle f'b_N, L_N(s) \rangle - \langle f'b, \mu_s \rangle ds + N \int_{0}^t \langle f''g_N^2h_N^2, L_N(s) \rangle ds \nonumber \\
	&\quad + \dfrac{N}{2} \int_{0}^t \iint \dfrac{f'(x) - f'(y)}{x - y} [NG_N (x, y) L_N(s)(dx) L_N(s)(dy) - G(x,y) \mu_s(dx) \mu_s(dy)] ds.\notag
\end{align}

The third term on the right-hand side of \eqref{measure error} can be written as
\begin{align*}
	&\quad N \int_{0}^t \langle f'b_N, L_N(s) \rangle - \langle f'b, \mu_s \rangle ds \nonumber \\
	&= N \int_{0}^t \langle f'b_N - f'b, L_N(s) \rangle ds + N \int_{0}^t \langle f'b, L_N(s) - \mu_s \rangle ds \nonumber \\
	&= N \int_{0}^t \langle f'b_N - f'b, L_N(s) \rangle ds + \int_{0}^t \mathcal{L}_s^N(f'b) ds.
\end{align*}
Thus, we have
\begin{align} \label{error of drift term}
	&\lim_{N\to\infty}\left| N \int_{0}^t \langle f'b_N, L_N(s) \rangle - \langle f'b, \mu_s \rangle ds - \int_{0}^t \mathcal{L}_s^N(f'b) ds \right|\notag\\
	\le& \lim_{N\to\infty} N \int_{0}^t \left| \langle f'b_N - f'b, L_N(s) \rangle \right| ds \nonumber \\
	\le&\lim_{N\to\infty} N T \|f'\|_{L^{\infty}(\mathbb{R})} \|b_N - b\|_{L^{\infty}(\mathbb{R})}=0.
\end{align}

For the fourth term on the right-hand side of \eqref{measure error}, 
\begin{align*}
	&N \int_{0}^t \langle f''g_N^2h_N^2, L_N(s) \rangle ds
	= \dfrac{N}{2} \int_{0}^t \langle f''(x)G_N(x,x), L_N(s) \rangle ds \\
	= &\dfrac{1}{2} \int_{0}^t \langle f''(x) (NG_N(x,x) - G(x,x)), L_N(s) \rangle ds + \dfrac{1}{2} \int_{0}^t \langle f''(x) G(x,x), L_N(s) \rangle ds.
\end{align*}
Hence, we have
\begin{align} \label{empirical measure term}
	& \left| N \int_{0}^t \langle f''g_N^2h_N^2, L_N(s) \rangle ds - \dfrac{1}{2} \int_{0}^t \langle f''(x) G(x,x), \mu_s \rangle ds \right| \nonumber \\
	\le& \left| \dfrac{1}{2} \int_{0}^t \langle f''(x) (NG_N(x,x) - G(x,x)), L_N(s) \rangle ds \right| + \left| \dfrac{1}{2} \int_{0}^t \langle f''(x) G(x,x), L_N(s) - \mu_s \rangle ds \right| \nonumber \\
	\le& \dfrac{1}{2} T \|f''\|_{L^{\infty}(\mathbb{R})} \|NG_N(x,x) - G(x,x)\|_{L^{\infty}(\mathbb{R}^2)} + \dfrac{1}{2} \left| \int_{0}^t \langle f''(x) G(x,x), L_N(s) - \mu_s \rangle ds \right|\notag \\
	&\longrightarrow  0,
\end{align}
as $N\to\infty$, where the last step follows from the weak convergence of $\{L_N(t), t \in [0,T]\}_{N \in \mathbb{N}}$ and the continuity and boundedness of  $G(x,x) f''(x)$  for $f \in \mathbb{F}$. 

The fifth term on the right-hand side of \eqref{measure error} can be written as
\begin{align} \label{double integral term}
	&\quad \dfrac{N}{2} \int_{0}^t \iint \dfrac{f'(x) - f'(y)}{x - y} [NG_N (x, y) L_N(s)(dx) L_N(s)(dy) - G(x,y) \mu_s(dx) \mu_s(dy)] ds \nonumber \\
	&= \dfrac{N}{2} \int_{0}^t \iint \dfrac{f'(x) - f'(y)}{x - y} [NG_N (x, y) - G(x,y)] L_N(s)(dx) L_N(s)(dy) ds \nonumber \\
	&\quad + \dfrac{N}{2} \int_{0}^t \iint \dfrac{f'(x) - f'(y)}{x - y} G(x,y) [L_N(s)(dx) - \mu_s(dx)] [L_N(s)(dy) - \mu_s(dy)] ds \nonumber \\
	&\quad + \dfrac{N}{2} \int_{0}^t \iint \dfrac{f'(x) - f'(y)}{x - y} G(x,y) \mu_s(dx) [L_N(s)(dy) - \mu_s(dy)] ds \nonumber \\
	&\quad + \dfrac{N}{2} \int_{0}^t \iint \dfrac{f'(x) - f'(y)}{x - y} G(x,y) [L_N(s)(dx) - \mu_s(dx)] \mu_s(dy) ds \nonumber \\
	&= \dfrac{N}{2} \int_{0}^t \iint \dfrac{f'(x) - f'(y)}{x - y} [NG_N (x, y) - G(x,y)] L_N(s)(dx) L_N(s)(dy) ds \nonumber \\
	&\quad + \dfrac{N}{2} \int_{0}^t \iint \dfrac{f'(x) - f'(y)}{x - y} G(x,y) [L_N(s)(dx) - \mu_s(dx)] [L_N(s)(dy) - \mu_s(dy)] ds \nonumber \\
	&\quad + \int_0^t \mathcal{L}_s^N\left( \int \dfrac{f'(x) - f'(y)}{x - y} G(x,y) \mu_s(dx) \right) ds,
\end{align}
where the last equality follows from the symmetry of  $\frac{f'(x) - f'(y)}{x - y} G(x,y)$. 
For the first term on the right-hand side of  \eqref{double integral term}, we have
\begin{align} \label{double integral subterm converge}
	&\lim_{N\to\infty} \left| \dfrac{N}{2} \int_{0}^t \iint \dfrac{f'(x) - f'(y)}{x - y} [NG_N (x, y) - G(x,y)] L_N(s)(dx) L_N(s)(dy) ds \right| \nonumber \\
	\le &\lim_{N\to\infty} \dfrac{NT}{2} \left\| \dfrac{f'(x) - f'(y)}{x - y} \right\|_{L^{\infty}(\mathbb{R}^2)} \|NG_N(x,y) - G(x,y)\|_{L^{\infty}(\mathbb{R}^2)}=0.
\end{align}

Therefore, by \eqref{eq-Q}, \eqref{measure error} and the above estimations \eqref{error of drift term}, \eqref{empirical measure term}, \eqref{double integral term}, and \eqref{double integral subterm converge}, we have that the term
\begin{align}
	&\quad Q_t^N(f) - N M_f^N(t)\notag \\
	&=\mathcal{L}_t^N(f) - \mathcal{L}_0^N(f) - N M_f^N(t) - \int_{0}^t \mathcal{L}_s^N(f'b) ds \notag\\
	&\quad- \dfrac{1}{2} \int_{0}^t \langle f''(x) G(x,x), \mu_s \rangle ds - \int_0^t \mathcal{L}_s^N\left( \int \dfrac{f'(x) - f'(y)}{x - y} G(x,y) \mu_s(dx) \right) ds\notag \\
	&\quad- \dfrac{N}{2} \int_{0}^t \iint \dfrac{f'(x) - f'(y)}{x - y} G(x,y) [L_N(s)(dx) - \mu_s(dx)] [L_N(s)(dy) - \mu_s(dy)] ds\label{eq-2.13'}
\end{align}
converges to $0$ almost surely as $N \rightarrow \infty$, uniformly in $t \in [0,T]$. Note that in \eqref{error of drift term}, \eqref{empirical measure term} and \eqref{double integral subterm converge}, the integrand function is bounded,  and hence the convergence is also in $L^p$ for all $p \ge 1$. Thus, $Q_t^N(f) - N M_f^N(t)$ with $f\in \mathbb F$ converges to $0$ in $L^p$ for all $p \ge 1$ uniformly in $t \in [0,T]$. 

Therefore, to prove the desired result, it suffices to show that, for any $k\in \mathbb N$ and $f_1, f_2, \dots, f_k\in \mathbb F$, the vector-valued stochastic process  $(N M_{f_1}^N(t), N M_{f_2}^N(t), \dots, N M_{f_k}^N(t))_{t\in[0,T]}$ converges in distribution to a centered Gaussian process $(G_t(f_1), G_t(f_2), \dots, G_t(f_k))_{t\in[0,T]}$ with covariance given by \eqref{cov}. To this end, by Lemma \ref{Rebolledo Thm} it suffices to prove that $\{ N M_f^N(t), t \in [0,T] \}_{N\in \mathbb N}$ are  martingales for $f \in \mathbb{F}$ such that  the following limit holds  in $L^1(\Omega)$,
\begin{align*}
	\lim_{N \rightarrow \infty} \langle N M_{f_1}^N, N M_{f_2}^N \rangle_t
	= 2 \int_0^t \langle f_1'(x) f_2'(x) G(x,x), \mu_s \rangle ds,  ~~\forall f_1, f_2\in \mathbb F. 
\end{align*}

 By the uniform convergence of $NG_N(x,y)$ towards $G(x,y)$, the boundedness of $f'(x)^2G(x,x)$ and \eqref{quadratic variation of martingale}, one can show that $\{ N M_f^N(t), t \in [0,T] \}_{N\in \mathbb N}$ are  martingales.  It follows from \eqref{martingale term} that, for $f_1, f_2 \in \mathbb{F}$,
\begin{align*}
	\langle N M_{f_1}^N, N M_{f_2}^N \rangle_t
	&= 4 \sum_{i=1}^N \int_{0}^t f_1'(\lambda_i^N(s)) f_2'(\lambda_i^N(s)) g_N^2(\lambda_i^N(s)) h_N^2(\lambda_i^N(s)) ds \\
	&= 4N \int_0^t \langle f_1' f_2' g_N^2 h_N^2, L_N(s) \rangle ds \\
	&= 2N \int_0^t \langle f_1'(x) f_2'(x) G_N(x,x), L_N(s) \rangle ds \\
	&= 2 \int_0^t \langle f_1'(x) f_2'(x) (NG_N(x,x) - G(x,x)), L_N(s) \rangle ds \\
	&\quad + 2 \int_0^t \langle f_1'(x) f_2'(x) G(x,x), L_N(s) \rangle ds.
\end{align*}
The term $\int_0^t \langle f_1'(x) f_2'(x) (NG_N(x,x) - G(x,x)), L_N(s) \rangle ds$ converges to 0 a.s. and in $L^p$ for all $p\ge 1$ due to the boundedness of $f_1'(x)$ and $f_2'(x)$ and the uniform convergence of $NG_N(x,y)$ towards $G(x,y)$. Furthermore, the following  convergence \begin{align*}
	\lim_{N \rightarrow \infty} \int_0^t \langle f_1'(x) f_2'(x) G(x,x), L_N(s) \rangle ds = \int_0^t \langle f_1'(x) f_2'(x) G(x,x), \mu_s \rangle ds,
\end{align*}
holds a.s. and in $L^p$ for all $p\ge1$, because of  the weak convergence of  $\{L_N(t), t \in [0,T]\}_{N \in \mathbb{N}}$ to $\{\mu_t, t\in[0,T]\}$ and the boundedness of $f_1'(x)f_2'(x)G(x,x)$. Therefore, 
$\langle N M_{f_1}^N, N M_{f_2}^N \rangle_t$ converges to $2 \int_0^t \langle f_1'(x) f_2'(x) G(x,x), \mu_s \rangle ds$ a.s. and in $L^p$ for all $p \ge 1$.

The proof is concluded. 
\end{proof}

If the eigenvalues in \eqref{eigenvalue SDE} are bounded,  the  test function space  $\mathbb{F}$ can be enlarged by removing the boundedness condition in the above theorem. 

\begin{corollary} \label{Thm for bounded general Wishart process} Assume the same conditions as in Theorem \ref{Thm-Wishart}.
Moreover, for  $T < \infty$, assume that
\begin{align} \label{bounded-eigen}
	\limsup_{N \rightarrow \infty} \sup_{t \in [0,T]} \max_{1 \le i \le N} \left| \lambda_i^N(t) \right| \le C(T),
\end{align}
a.s. for some constant $C(T)$ depending on $T$. Then Theorem \ref{Thm-Wishart} still holds if the set $\mathbb{F}$ of  test function  is replaced by $C^2(\mathbb{R})$.
\end{corollary}

\begin{proof}
It follows from \eqref{bounded-eigen} that  all but finitely many terms in $\{ \sup_{t \in [0,T]} \max_{1 \le i \le N} |\lambda_i^N(t)| \}_{N \in \mathbb{N}}$ are bounded by $A(T)=C(T) + 1$ a.s.. Thus there is a measurable set $A\subset \Omega$ with $\mathbb P(A)=1$  and a random variable $N_0\in \mathbb N$, such that for $\omega \in A$,  the empirical measure $\{ L_N(t)(\omega), t \in [0,T] \}$ is supported in $[-A(T), A(T)]$ for all $N\ge N_0(\omega)$. Hence the limit $\{\mu_t, t\in[0,T]\}$ also has the same support.  By \cite[1.46]{Rudin1991}, there exists a cut-off function $\eta(x) \in C^{\infty}(\mathbb{R})$ equal to $1$  on $[-A(T), A(T)]$, of which the  support is $[-2A(T), 2A(T)]$.  If we replace $f$ by $f\eta$, noting that  $f \eta \in \mathbb{F}$ for $f \in C^2(\mathbb{R})$ and that $f\eta=f$ on $[-A(T), A(T)]$, we can show that the term $Q_t^N(f) - N M_f^N(t)$ in \eqref{eq-2.13'} converges to 0 a.s. using the same argument as in the proof of Theorem \ref{Thm-Wishart}. Then following the rest part of the proof, it is easy to get the result of Theorem \ref{Thm-Wishart}.
\end{proof}

\begin{remark}
Under the conditions in Theorem \ref{Thm-Wishart}, \eqref{bounded-eigen} yields the almost sure convergence of $Q_t^N(f) - N M_f^N(t)$ towards $0$ for $f\in C^2(\R)$. The next Corollary provides a sufficient condition for the $L^p$ convergence for $p\ge 1$. 
\end{remark}

\begin{corollary} \label{Coro-Lp}  Assume the same conditions as in Theorem \ref{Thm-Wishart}.
For $T < \infty$, for all $p \ge 1$ and  all $N \ge cp$ for some positive constant  $c$, assume that
\begin{align}\label{eq-p-bound}
	\mathbb{E} \left[ \sup_{t \in [0,T]} \langle |x|^p, L_N(t) \rangle \right] \le C(T)^p,
\end{align}
where $C(T)$ is a positive constant  depending only on $T$.  Furthermore, assume that $G(x,x)$ and its derivative  have at most polynomial growth. Then for $f\in C^3(\mathbb{R})$ of which the derivatives have at most polynomial growth,  $Q_t^N(f) - N M_f^N(t)$ converges to $0$ in $L^p$  uniformly in $t \in [0,T]$ for all $p\ge 1$. 

As a consequence, Theorem \ref{Thm-Wishart} holds for such test functions $f$.
\end{corollary}

\begin{proof}
By the analysis in the proof of Theorem \ref{Thm-Wishart},  it suffices to show
\begin{align} \label{eq-2.14}
	\limsup_{N \rightarrow \infty} \mathbb{E} \left[ \sup_{t \in [0,T]} \left| \langle g, L_N(t) \rangle - \langle g, \mu_t \rangle \right|^p \right] = 0,
\end{align}
for $p \ge 1$ and $g\in C^1(\mathbb{R})$ with $ |g'(x)|\le a|x|^{n-1} + b$ for some $a,b \in \mathbb{R}, \ n \in \mathbb{N}_+$.  More precisely, one can check that under the conditions \eqref{eq-p-bound} and \eqref{eq-2.14}, the convergences to 0 in \eqref{error of drift term}, \eqref{empirical measure term} and \eqref{double integral subterm converge} are uniform in $L^p$, and hence $Q_t^N(f) - N M_f^N(t)$ in  \eqref{eq-2.13'} converges to 0 in $L^p$ uniformly.

By Markov inequality and \eqref{eq-p-bound},
\begin{align*}
	\mathbb{P} \left( \sup_{t \in [0,T]} \max_{1 \le i \le N} |\lambda_i^N(t)| > C(T) + 1 \right)
	&\le (C(T) + 1)^{-p} \mathbb{E} \left[ \sup_{t \in [0,T]} \max_{1 \le i \le N} |\lambda_i^N(t)|^p \right] \\
	&\le (C(T) + 1)^{-p} N \mathbb{E} \left[ \sup_{t \in [0,T]} \langle |x|^p, L_N(t) \rangle \right] \\
	&\le N \left( \dfrac{C(T)}{C(T) + 1} \right)^p.
\end{align*}
Choosing $p = \ln^2 N$, we have
\begin{align*}
	\sum_{N=1}^{\infty} \mathbb{P} \left( \sup_{t \in [0,T]} \max_{1 \le i \le N} |\lambda_i^N(t)| > C(T) + 1 \right)
	&\le \sum_{N=1}^{\infty} N \left( \dfrac{C(T)}{C(T) + 1} \right)^p \\
	&= \sum_{N=1}^{\infty} N^{1 + \ln N \ln \frac{C(T)}{C(T)+1}} \\
	&<\infty.
\end{align*}
By Borel-Cantelli lemma, we get that almost surely,
\begin{align*}
	\limsup_{N \rightarrow \infty} \sup_{t \in [0,T]} \max_{1 \le i \le N} \left| \lambda_i^N(t) \right| \le C(T)+1.
\end{align*}
By the proof of Corollary \ref{Thm for bounded general Wishart process}, the limit measure $\{\mu_t, t\in [0,T]\}$ is supported in $[-C(T)-1, C(T)+1]$.

For $g\in C^1(\mathbb{R})$ with $ |g'(x)|\le a|x|^{n-1} + b$   for some $a,b \in \mathbb{R}, \ n \in \mathbb{N}_+$, define 
\begin{align*}
	g_{\delta} (x) = g \left( \dfrac{x}{1 + \delta x^2} \right)
\end{align*}
for $\delta>0.$ Then $g_{\delta}(x)$ is a bounded continuous function, and hence
\begin{align*}
	\lim_{N \rightarrow \infty} \sup_{t \in [0,T]} |\langle g_{\delta}, L_N(t) \rangle - \langle g_{\delta},\mu_t \rangle| = 0,
\end{align*}
almost surely. By dominated convergence theorem,
\begin{align} \label{eq-2.15}
	\lim_{N \rightarrow \infty} \mathbb{E} \left[ \sup_{t \in [0,T]} |\langle g_{\delta}, L_N(t) \rangle - \langle g_{\delta},\mu_t \rangle|^p \right] = 0.
\end{align}

Note that $g'(x)$ grows no faster than polynomials of degree $n-1$, by the mean value theorem, it is not difficult to show  $|g(x) - g_{\delta}(x)| \le C \delta (|x|^{n+2} + |x|^3)$, which implies that $g_{\delta}$ converges to $g$ uniformly in any compact interval as $\delta \rightarrow 0^+$. Thus,
\begin{align} \label{eq-2.16}
	\lim_{\delta \rightarrow 0^+} \sup_{t \in [0,T]} \left| \langle g, \mu_t \rangle - \langle g_{\delta}, \mu_t \rangle \right| = 0.
\end{align}

Finally, by the Jensen's inequality and \eqref{eq-p-bound}, we obtain that,  as $\delta \to 0^+$,
\begin{align} \label{eq-2.17}
	&\quad \mathbb{E} \left[ \sup_{t \in [0,T]} |\langle g, L_N(t) \rangle - \langle g_{\delta}, L_N(t) \rangle|^p \right] \nonumber \\
	&\le C^p \delta^p \mathbb{E} \left[ \sup_{t \in [0,T]} |\langle |x|^{n+2}+|x|^3, L_N(t) \rangle|^p \right] \nonumber \\
	&\le  C^p\delta^p \mathbb{E} \left[ \sup_{t \in [0,T]} |\langle (|x|^{n+2}+|x|^3)^p, L_N(t) \rangle|\right]\nonumber\\
	&\le 2^p C^p (C(T)^{(n+2)p} + C(T)^{3p}) \delta^p\to 0
\end{align}
uniformly in $N\in \mathbb N_+$.

By \eqref{eq-2.15}, \eqref{eq-2.16}, \eqref{eq-2.17} and the triangle inequality, we can obtain \eqref{eq-2.14}, and the proof is concluded.
\end{proof}

\begin{proposition} \label{Prop on linearity of the Gaussian process}
Consider the centered Gaussian family $\{G_t(f), f\in \mathbb F\}$ in Theorem \ref{Thm-Wishart} with covariance
\begin{align*} 
	\mathbb{E} \left[ G_t(f) G_t(g) \right]
	= 2 \int_0^{t} \langle f'(x) g'(x) G(x,x), \mu_u \rangle du, ~~\forall f, g\in\mathbb F.
\end{align*}
We have the following linear property, for $f_1, f_2 \in \mathbb F$ and $\alpha_1, \alpha_2 \in \R$,
\begin{align} \label{linearity of the Gaussian process}
	G_t(\alpha_1 f_1 + \alpha_2 f_2) = \alpha_1 G_t(f_1) + \alpha_2 G_t(f_2), \quad \forall t \in [0,T],
\end{align}
almost surely.
\end{proposition}

\begin{proof}
For $f_1, f_2 \in \mathbb{F}$ and $\alpha_1, \alpha_2 \in \R$, it is easy to check that $\alpha_1 f_1 + \alpha_2 f_2 \in \mathbb{F}$. By the proof of Theorem \ref{Thm-Wishart}, the random vector $(N M_{f_1}^N(t), N M_{f_2}^N(t), N M_{\alpha_1 f_1 + \alpha_2 f_2}^N(t))_{t \in [0,T]}$ converges in distribution to $(G_t(f_1), G_t(f_2), G_t(\alpha_1 f_1 + \alpha_2 f_2))_{t \in [0,T]}$. Hence,  the linear combination $(\alpha_1 N M_{f_1}^N(t) + \alpha_2 N M_{f_2}^N(t) - N M_{\alpha_1 f_1 + \alpha_2 f_2}^N(t))_{t \in [0,T]}$ converges in distribution to $(\alpha_1 G_t(f_1) + \alpha_2 G_t(f_2) - G_t(\alpha_1 f_1 + \alpha_2 f_2))_{t \in [0,T]}$.

By \eqref{martingale term}, we can see that the martingale $M_f^N(t)$ is linear with respect to the function $f$, so $\alpha_1 N M_{f_1}^N(t) + \alpha_2 N M_{f_2}^N(t) = N M_{\alpha_1 f_1 + \alpha_2 f_2}^N(t)$ for all $t \in [0,T]$ and all $N \in \mathbb{N}$, which implies that the process $(\alpha_1 N M_{f_1}^N(t) + \alpha_2 N M_{f_2}^N(t) - N M_{\alpha_1 f_1 + \alpha_2 f_2}^N(t))_{t \in [0,T]}$ is actually a zero process. Thus, as the limit of the convergence in distribution, $(\alpha_1 G_t(f_1) + \alpha_2 G_t(f_2) - G_t(\alpha_1 f_1 + \alpha_2 f_2))_{t \in [0,T]}$ is also a zero process, which implies \eqref{linearity of the Gaussian process}.
\end{proof}

\subsection{Central limit theorem for particle systems}
\label{sec:particle}
In this subsection, we provide the central limit theorem for the empirical measure of the following particle system: for $1\le i\le N$,
\begin{align} \label{SDE-particle}
	dx_i^N(t) = \sigma^N(x_i^N(t)) dW_i(t)+ \left( b_N(x_i^N(t)) + \sum_{j:j\neq i} \dfrac{H_N(x_i^N(t), x_j^N(t))}{x_i^N(t) - x_j^N(t)} \right) dt, \ t \ge 0,
\end{align}
with $H_N(x,y)$ being a symmetric function. This particle system was introduced in \cite{Graczyk2014} as a generalization of \eqref{eigenvalue SDE}.  Under proper conditions, the existence and
uniqueness of the non-colliding strong solution was obtained in \cite{Graczyk2014}, and it was shown in \cite{Song2019} that the family of empirical measure $\{L_N(t), t \in [0,T]\}$ is tight almost surely, and any limit $\{\mu_t, t \in [0,T]\}$ satisfies
\begin{align*} 
	\int \dfrac{\mu_t(dx)}{z-x}
	=& \int \dfrac{\mu_0(dx)}{z-x} + \int_{0}^t \left[ \int \dfrac{b(x)}{(z-x)^2} \mu_s(dx) \right] ds +\int_{0}^t \left[ \int \dfrac{\sigma(x)^2}{(z-x)^3} \mu_s(dx) \right] ds\notag \\
	&+ \int_{0}^t \left[ \iint \dfrac{H(x,y)}{(z-x) (z-y)^2} \mu_s(dx) \mu_s(dy) \right] ds, ~~ \forall z\in \mathbb C\setminus \R,
\end{align*}
where, $b(x)$, $\sigma(x)$ and $H(x,y)$ are the uniform limits of $b_N(x)$, $\sigma^N(x)$ and $NH_N(x,y)$, respectively.

Now we adopt the following set of test functions
\begin{align}
	\widetilde{\mathbb F}
	&= \Bigg\{ f \in C_b^2(\mathbb{R}): \| f'(x) b(x) \|_{L^{\infty}(\mathbb{R})} < \infty, ~~\left\| \dfrac{f'(x) - f'(y)}{x-y} H(x,y) \right\|_{L^{\infty}(\mathbb{R}^2)} < \infty \notag\\ &\qquad \| f'(x) \tilde{\sigma}(x) \|_{L^{\infty}(\mathbb{R})} < \infty, ~~ \| f''(x) \tilde{\sigma}^2(x) \|_{L^{\infty}(\mathbb{R})} < \infty \Bigg\},\label{space F'}
\end{align}
where $\tilde{\sigma}(x)$ is the uniform limit of $\sqrt{N} \sigma^N(x)$. Considering the centered fluctuation process, for $f \in \widetilde{\mathbb F}$, 
\begin{align*}
	\widetilde{Q}_t^N(f) &= \mathcal{L}_t^N(f) - \mathcal{L}_0^N(f) - \int_{0}^t \mathcal{L}_s^N(f'b) ds - \dfrac{1}{2} \int_{0}^t \langle f''(x) (\tilde{\sigma}^2(x) - H(x,x)), \mu_s \rangle ds \\
	&\quad- \int_0^t \mathcal{L}_s^N\left( \int \dfrac{f'(x) - f'(y)}{x - y} H(x,y) \mu_s(dx) \right) ds \\
	&\quad- \dfrac{N}{2} \int_{0}^t \iint \dfrac{f'(x) - f'(y)}{x - y} H(x,y) [L_N(s)(dx) - \mu_s(dx)] [L_N(s)(dy) - \mu_s(dy)] ds,
\end{align*}
as an extension of Theorem \ref{Thm-Wishart}, we have the following result. 
\begin{theorem} \label{Thm-particle}
Suppose that the limit functions $\tilde{\sigma}(x)$, $b(x)$ and $H(x,y)$ are continuous and the following conditions hold, 
\begin{align}\label{conv-speed-particle system}
\begin{aligned}
	\lim_{N \rightarrow \infty} N \|b_N(x) - b(x)\|_{L^{\infty}(\mathbb{R})} = 0, \\
	\lim_{N \rightarrow \infty} N \|NH_N(x,y) - H(x,y)\|_{L^{\infty}(\mathbb{R}^2)} = 0.
\end{aligned}
\end{align}
Also assume that \eqref{SDE-particle} has a non-exploding and non-colliding strong solution, such that the sequence of the empirical measures  $\{L_N(t), t \in [0,T]\}_{N \in \mathbb{N}}$ converges weakly to $\{\mu_t, t\in [0, T]\}$.

Then, for any $k \in \mathbb{N}$ and any $f_1, \ldots, f_k \in \widetilde{\mathbb F}$, $(\widetilde{Q}_t^N(f_1), \ldots, \widetilde{Q}_t^N(f_k))_{t \in [0,T]}$ converges in distribution to a centered Gaussian process $(\widetilde G_t(f_1), \ldots, \widetilde G_t(f_k))_{t \in [0,T]}$ with covariance 
\begin{align*}
	\mathbb{E} \left[ \widetilde G_t(f_i) \widetilde G_s(f_j) \right]
	= \int_0^{t \wedge s} \langle f_i'(x) f_j'(x) \tilde{\sigma}^2(x), \mu_u \rangle du, ~~1 \le i,j \le k.
\end{align*}
\end{theorem}

Results analogous to Corollary \ref{Thm for bounded general Wishart process}, Corollary \ref{Coro-Lp} and Proposition \ref{Prop on linearity of the Gaussian process} are as follows.

\begin{corollary} \label{Thm for bounded particle system}
 Assume the same conditions as in Theorem \ref{Thm-particle}.
Moreover, for $T < \infty$, assume that
\begin{align*}
	\limsup_{N \rightarrow \infty} \sup_{t \in [0,T]} \max_{1 \le i \le N} \left| x_i^N(t) \right| \le C(T),
\end{align*}
almost surely for some constant $C(T)$ depending on $T$. Then Theorem \ref{Thm-particle} still holds if the set $\,\widetilde{\mathbb{F}}$ of test function  is replaced by $C^2(\mathbb{R})$.
\end{corollary}

{
\begin{corollary} \label{Coro-Lp-particle}  Assume the same conditions as in Theorem \ref{Thm-particle}. For $T < \infty$ and all $p \ge 1$, assume that
\begin{align*}
	\mathbb{E} \left[ \sup_{t \in [0,T]} \langle |x|^p, L_N(t) \rangle \right] \le C(T)^p,
\end{align*}
for some positive constant $C(T)$ which depends only on $T$. Furthermore, assume that $(\frac{1}{2} \tilde{\sigma}(x)^2 - H(x,x)) f''(x)$ and its derivative have at most polynomial growth.  Then for $f\in C^3(\mathbb{R})$ of which the derivatives have at most polynomial growth, $\widetilde{Q}_t^N(f) - N M_f^N(t)$ converges to $0$ in $L^p$ for all $p \ge 1$ uniformly in $t \in [0,T]$.
	\end{corollary}
}

\begin{proposition} \label{Prop on linearity of the Gaussian process-particle system}
Consider the centered Gaussian family $\{\widetilde G_t(f), f\in \widetilde{\mathbb F}\}$ with covariance
\begin{align*} 
	\mathbb{E} \left[\widetilde G_t(f) \widetilde G_t(g) \right]
	= \int_0^{t} \langle f'(x) g'(x) \tilde\sigma^2(x), \mu_u \rangle du, ~~\forall f, g\in \widetilde{\mathbb F}. 
\end{align*}
We have the following linear property, for $f_1, f_2 \in {\widetilde{\mathbb F}}$ and $\alpha_1, \alpha_2 \in \R$,
\begin{align*}
	\widetilde G_t(\alpha_1 f_1 + \alpha_2 f_2) = \alpha_1 \widetilde G_t(f_1) + \alpha_2 \widetilde G_t(f_2), \quad \forall t \in [0,T],
\end{align*}
almost surely.
\end{proposition}

The proofs of Theorem \ref{Thm-particle}, Corollary \ref{Thm for bounded particle system}, Corollary \ref{Coro-Lp-particle} and Proposition \ref{Prop on linearity of the Gaussian process-particle system} are similar to those of Theorem \ref{Thm-Wishart}, Corollary \ref{Thm for bounded general Wishart process}, Corollary \ref{Coro-Lp} and Proposition \ref{Prop on linearity of the Gaussian process}, respectively, and thus omitted.

\section{Applications}\label{section-application}

In this section, we apply our main results obtained in Section \ref{section-CLT} to the eigenvalues of Wishart process (Section \ref{sec-wishart}), the Dyson's Brownian motion (Section \ref{sec-dyson}) and  the eigenvalues of symmetric Ornstein-Uhlenbeck matrix process (Section \ref{sec-ou}). In particular, for these three cases, we will show the boundedness of the moments of the empirical measures assuming proper initial conditions. This enables us to apply  Corollaries \ref{Thm for bounded general Wishart process}, \ref{Coro-Lp}, \ref{Thm for bounded particle system} and \ref{Coro-Lp-particle} to study the flunctuations $\L_t(f)$ for   polynomial functions $f\in \R[x]$, and recursive formulas are obtained for the basis $ \{\L_t(x^n), t\in[0,T]\}_{ n\in\mathbb N}$ of $\{\L_t(f), t\in[0,T]\}_{f\in \R[x]}$.  Note that these results are more precise than the general results in Section \ref{section-CLT}, where we study the centered process $\{Q_t^N(f)\}$ for more restricted test function $f$.

\subsection{Comparison principle} \label{section-comparison}

In this subsection, we provide a comparison principle  for SDE \eqref{eigenvalue SDE} and particle system \eqref{SDE-particle}, which allows us to obtain the boundedness of the eignenvalues/particles under more general initial conditions in Sections \ref{sec-wishart}, \ref{sec-dyson} and \ref{sec-ou}. 

 Throughout this subsection, the dimension $N$ is fixed and thus subscripts/superscripts are removed. Precisely, consider the following two particle systems: for $1 \le i \le N, \ t \ge 0$,
\begin{equation} \label{Compare-particle sde 1}
\begin{cases}
	dx_i(t) = \sigma_i(x_i(t)) dW_i(t)+ \left( b_i(x_i(t)) + \sum_{j:j\neq i} \dfrac{H_{ij}(x_i(t), x_j(t))}{x_i(t) - x_j(t)} \right) dt, \\
	x_1(t) \le \ldots \le x_N(t),
\end{cases}
\end{equation}
and
\begin{equation} \label{Compare-particle sde 2}
\begin{cases}
	dy_i(t) = \sigma_i(y_i(t)) dW_i(t)+ \left( \tilde{b}_i(y_i(t)) + \sum_{j:j\neq i} \dfrac{H_{ij}(y_i(t), y_j(t))}{y_i(t) - y_j(t)} \right) dt, \\
	y_1(t) \le \ldots \le y_N(t),
\end{cases}
\end{equation}
with non-colliding initial values $x(0) = (x_1(0), \ldots, x_N(0))$ and $y(0) = (y_1(0), \ldots, y_N(0))$, respectively. Here, the functions $\sigma_i(x)$, $b_i(x)$ and $\tilde{b}_i(x)$ for $1\le i \le N$ are continuous, and  $H_{ij}(x,y)$ with $i\neq j$ is a continuous,  non-negative and symmetric function satisfying the condition \cite[(A1)]{Graczyk2014}:
\begin{align} \label{compare-particle force condition}
	\dfrac{H_{ij}(w,z)}{z-w} \le \dfrac{H_{ij}(x,y)}{y-x}, \quad \forall w < x < y < z, ~~ 1 \le i \neq j \le N.
\end{align}

Note that conditions for the existence and uniqueness of a non-colliding and non-exploding strong solution to \eqref{Compare-particle sde 1} (or \eqref{Compare-particle sde 2}) were obtained in \cite{Graczyk2014}.  In particular, under conditions (A2) - (A5) therein, the particles will separate from each other immediately after starting from a colliding initial state, and will not collide forever. 

\begin{theorem} \label{Thm-compare of particle}
Suppose $x(t) = (x_1(t), \ldots, x_N(t))$ and $y(t) = (y_1(t), \ldots, y_N(t))$ are the non-exploding and non-colliding unique strong solutions to \eqref{Compare-particle sde 1} and \eqref{Compare-particle sde 2}, respectively. Assume that there exists a strictly increasing function $\rho: [0,\infty) \rightarrow [0,\infty)$ with $\rho(0) = 0$ and
\begin{align*}
	\int_{0+} \rho^{-2}(u) du = \infty,
\end{align*}
such that
\begin{align*}
	|\sigma_i(u) - \sigma_i(v)| \le \rho(|u - v|),   ~~ \forall u,v \in \mathbb{R}, ~1 \le i \le N.
\end{align*}
If  we further assume that $b_i(u) \le \tilde{b}_i(u)$  for all $u \in \mathbb{R}$,  and $x_i(0) \le y_i(0)$ a.s., $1 \le i \le N$, then
\begin{align*}
	\mathbb{P} \left( x_i(t) \le y_i(t), \forall t \ge 0, 1 \le i \le N \right) = 1.
\end{align*}
\end{theorem}

\begin{proof}
The continuity of the functions $H_{ij}$ and the condition \eqref{compare-particle force condition} implies that for all $1 \le i \neq j \le N$,
\begin{align*}
	\dfrac{H_{ij}(x,z)}{x-z} \ge \dfrac{H_{ij}(x,y)}{x-y}, \quad \forall x < y \le z,
\end{align*}
and
\begin{align*}
	\dfrac{H_{ij}(w,y)}{y-w} \le \dfrac{H_{ij}(x,y)}{y-x}, \quad \forall w \le x < y.
\end{align*}
Hence, the drift functions
\begin{align*}
	F(u) = \left( b_i(u_i) + \sum_{j:j\neq i} \dfrac{H_{ij}(u_i, u_j)}{u_i - u_j} \right)_{1 \le i \le N}, \quad
	\widetilde{F}(u) = \left( \tilde{b}_i(u_i) + \sum_{j:j\neq i} \dfrac{H_{ij}(u_i, u_j)}{u_i - u_j} \right)_{1 \le i \le N},
\end{align*}
satisfy the quasi-monotonously increasing condition in Lemma \ref{comparison}.

   In order to apply Lemma \ref{comparison} to get the desired result, we use an approximation argument to remove the singularities of the drift functions $F$ and $\widetilde{F}$.  For $\epsilon > 0$, let 
\begin{align*}
	\Delta_{\epsilon} = \left\{ u = (u_1, \ldots, u_N) \in \mathbb{R}^N : \ \forall 1 \le i \le N-1, u_{i+1} - u_{i} > \epsilon \right\}
\end{align*}
and define the stopping time
\begin{align*}
	\tau_{\epsilon} = \inf_{t > 0} \left\{ \min_{1 \le i \le N-1}  (x_{i+1}(t) - x_i(t)) \wedge (y_{i+1}(t) - y_i(t)) \le \epsilon \right\}.
\end{align*}
One can find continuous quasi-monotonously increasing functions $F_{\epsilon}$ and $\widetilde F_\epsilon$, such that they coincide with $F$ and $\widetilde F$ in $\Delta_{\epsilon}$, repspectively. Before time $\tau_{\epsilon}$, both $x$-particles and $y$-particles stay in $\Delta_{\epsilon}$ and thus satisfy   \eqref{Compare-particle sde 1} and \eqref{Compare-particle sde 2} with  drift functions $\widetilde F_\epsilon$ and  $\widetilde F_\epsilon$, respectively.

Applying Lemma \ref{comparison} to the processes $x^{\epsilon}$ and $y^{\epsilon}$, we have
\begin{align*}
	\mathbb{P} \left( x^{\epsilon}_i(t) \le y^{\epsilon}_i(t), \ \forall \ t \ge 0, 1 \le i \le N \right) = 1,
\end{align*}
which implies
\begin{align*}
	\mathbb{P} \left( x_i(t) \le y_i(t), \ \forall  t \in[0, \tau_{\epsilon}], 1 \le i \le N \right) = 1.
\end{align*}
The desired result now follows from the non-colliding property $\lim_{\epsilon\to 0^+}\tau_{\epsilon} = \infty$.

\end{proof}

As a corollary of Theorem \ref{Thm-compare of particle}, we have the following comparison principle for SDE \eqref{eigenvalue SDE} of eigenvalue processes. Note that  the existence and uniqueness of the non-colliding and non-exploding strong solution was obtained under proper conditions  in \cite{Graczyk2013}.
\begin{corollary} \label{Thm-compare of eigenvalue}
Suppose that the following systems of eigenvalue SDEs
\begin{align*}
	&d \lambda_i(t)
	= 2 g_N(\lambda_i(t)) h_N(\lambda_i^N(t)) dW_i(t) + \left( b_N(\lambda_i^N(t)) + \sum_{j:j \neq i} \dfrac{G_N(\lambda_i(t),\lambda_j(t))}{\lambda_i(t) - \lambda_j(t)} \right) dt, \ 1 \le i \le N, \\
	&\lambda_1(t) \le \ldots \le \lambda_N(t), \ t \ge 0,
\end{align*}
and
\begin{align*}
	&d \theta_i(t)
	= 2 g_N(\theta_i(t)) h_N(\theta_i^N(t)) dW_i(t) + \left( \tilde{b}_N(\theta_i^N(t)) + \sum_{j:j \neq i} \dfrac{G_N(\theta_i(t),\theta_j(t))}{\theta_i(t) - \theta_j(t)} \right) dt, \ 1 \le i \le N, \\
	&\theta_1(t) \le \ldots \le \theta_N(t), \ t \ge 0,
\end{align*}
with non-colliding initial values $\lambda(0) = (\lambda_1(0), \ldots, \lambda_N(0))$ and $\theta(0) = (\theta_1(0), \ldots, \theta_N(0))$, respectively, have  non-exploding and non-colliding unique strong solutions $\lambda(t) = (\lambda_1(t), \ldots, \lambda_N(t))$ and $\theta(t) = (\theta_1(t), \ldots, \theta_N(t))$, respectively. Here, $g_N(x)$, $h_N(x)$, $b_N(x)$ and $\tilde{b}_N(x)$ are continuous functions, and $G_N(x,y) = g_N^2(x) h_N^2(y) + g_N^2(y) h_N^2(x)$ satisfies
\begin{align} \label{compare-eigenvalue force condition}
	\dfrac{G_N(w,z)}{z-w} \le \dfrac{G_N(x,y)}{y-x}, \quad \forall w < x < y < z.
\end{align}
Assume that there exists a strictly increasing function $\rho: [0,\infty) \rightarrow [0,\infty)$ with $\rho(0) = 0$ and
\begin{align*}
	\int_{0^+} \rho^{-2}(u) du = \infty,
\end{align*}
such that
\begin{align*}
	|g_N(u)h_N(u) - g_N(v)h_N(v)| \le \rho(|u - v|), \quad \forall u,v \in \mathbb{R}.
\end{align*}
Furthermore, we assume that $b_N(u) \le \tilde{b}_N(u)$ for all $u \in \mathbb{R}$. If $\lambda_i(0) \le \theta_i(0)$ for all $1 \le i \le N$ almost surely, then
\begin{align*}
\mathbb{P} \left( \lambda_i(t) \le \theta_i(t), \forall t \ge 0, 1 \le i \le N \right) = 1.
\end{align*}
\end{corollary}

\subsection{Application to eigenvalues of Wishart process}\label{sec-wishart}

In this subsection, we discuss the limit theorem for the Wishart process.  As illustrated in \cite{Graczyk2013} and \cite{Song2019}, the scaled Wishart process $X^N_t = \tilde{B}^\intercal(t) \tilde{B}(t) / N$, where $\tilde{B}(t)$ is a $P \times N$ Brownian matrix with $P > N-1$, is the solution to \eqref{matrix SDE} with the coefficient functions
\begin{align*}
	g_N(x) h_N(y) = \dfrac{\sqrt{x}}{\sqrt{N}}, \quad
	b_N(x) = \dfrac{P}{N}.
\end{align*}
The eigenvalue processes  now satisfy
\begin{align} \label{Wishart eigenvalue SDE}
	d \lambda_i^N(t)
	= 2 \dfrac{\sqrt{\lambda_i^N(t)}}{\sqrt{N}} dW_i(t) + \left( \dfrac{P}{N} + \dfrac{1}{N} \sum_{j:j \neq i} \dfrac{\lambda_i^N(t) + \lambda_j^N(t)}{\lambda_i^N(t) - \lambda_j^N(t)} \right) dt, ~ 1\le i \le N, ~ t\ge 0.
\end{align}
In this case, we have
\begin{align}\label{eq-3.6}
	NG_N(x,y) = G(x,y) = x+y\ \ \text{ and } \ \
	b(x) = \lim_{N \rightarrow \infty} \dfrac{P}{N} = c \ge 1.
\end{align}

By \cite[Theorem 3]{Graczyk2019}, all the components of the solution to \eqref{Wishart eigenvalue SDE} are non-negative if all the components of the initial value are non-negative.  Let $\mathbb{P}^N$ be the distribution on $\Delta_N = \{x=(x_1, x_2, \dots, x_N) \in \mathbb{R}^N: 0 < x_1 < \ldots < x_N \}$ with density
\begin{align} \label{density of Wishart eigenvalue}
	p(x) = C_{N,p} \exp \left( -\dfrac{N}{2} \sum_{i=1}^N x_i \right) \prod_{i=1}^N x_i^{(P-N-1)/2} \prod_{1\le j<i\le N} (x_i - x_j),
\end{align}
where $C_{N,p} > 0$ is a normalization constant. Then we have the following estimation on the eigenvalues.

\begin{lemma} \label{largest eigenvalue bound}
Let  $\xi^N=(\xi^N_1, \dots, 
\xi^N_N)$ be a random vector that is independent of $(W_1, \dots, W_N)$ and has  \eqref{density of Wishart eigenvalue} as its joint probability density function. Assume that $(\lambda_1^N(0), \dots, \lambda_N^N(0))$ is independent of $(W_1, \dots, W_N)$ and that  there exists a constant $a > 0$, such that $\lambda_i^N(0) \le a \xi_i^N$ for $1 \le i \le N$ almost surely. Then there exists a stationary stochastic process $u^N(t)$ with initial value $u^N(0) = \xi^N$ satisfying, for $1 \le i \le N$ and $t\ge 0$,
\begin{align*}
	\lambda_i^N(t) \le v_i^N(t)
	= (t+a) u_i^N(t).
\end{align*}
\end{lemma}

\begin{proof}
Consider the following system of SDEs, for $1\le i\le N$,
\begin{align} \label{self-similarity sde}
	d u_i^N(t)
	= 2 \dfrac{\sqrt{u_i^N(t)}}{\sqrt{N(t+a)}} dW_i(t) + \dfrac{1}{t+a} \left( \dfrac{P}{N} - u_i^N(t) + \dfrac{1}{N} \sum_{j:j \neq i} \dfrac{u_i^N(t) + u_j^N(t)}{u_i^N(t) - u_j^N(t)} \right) dt, \quad t \ge0,
\end{align}
with initial value $u_i^N(0)=\xi_i^N(0)$ distributed according to $\mathbb{P}^N$ and $u_1^N(t) \le \ldots \le u_N^N(t)$. 

 Note that the pathwise uniqueness proved in \cite[Theorem 2]{Graczyk2013} is still valid if the coefficient functions depend on the time $t$ and  the corresponding conditions therein hold uniformly in $t$. Furthermore, the boundedness estimation and the McKean's argument in \cite[Theorem 5]{Graczyk2013} is also valid when $t \ge 0$. Therefore, the system of SDEs \eqref{self-similarity sde} has a  unique non-colliding strong solution. 

If at any time $t$, $u^N(t)$ has the distribution $\mathbb{P}^N$, then Lemma \ref{Lemma-stationary distribution} yields that $\frac{d}{dt} \mathbb{E} [f(u^N(t))]$ vanishes for $f \in C_b^2(\mathbb{R})$. Since $u^N(0)$ is distributed according to $\mathbb{P}^N$, we can conclude that $(u^N(t))_{t \ge 0}$ is a stationary process with marginal distribution $\mathbb{P}^N$.

Now let $v_i^N(t) = (t+a) u_i^N(t)$ for $1 \le i \le N$ and $v^N(t) = (v_1^N(t), \ldots, v_N^N(t))$. Then the It\^{o} formula shows that $v^N(t)$ is a solution to \eqref{Wishart eigenvalue SDE} with initial value $v^N(0) = a u^N(0) = a \xi^N$. Noting  that the solution of \eqref{Wishart eigenvalue SDE} is non-negative and that $G_N(x,y) = (x+y)/N$ with non-negative variables satisfies  condition \eqref{compare-eigenvalue force condition},  we can apply the comparison principle in Corollary \ref{Thm-compare of eigenvalue} to obtain
\begin{align*}
	\lambda_i^N(t) \le v_i^N(t)
	= (t+a) u_i^N(t).
\end{align*}
The proof is concluded.
\end{proof}

\begin{lemma} \label{largest eigen uniform bound}
Assume the same conditions as in Lemma \ref{largest eigenvalue bound}. Then for any $T < \infty$, there exists a positive constant $C(a,T)$  depending only on $(a,T)$, such that for all $p \ge 1$,
\begin{align*}
	\mathbb{E} \left[ \sup_{t \in [0,T]} \langle |x|^p, L_N(t) \rangle \right] \le C(a,T)^p,
\end{align*}
almost surely for $N \ge (2p-1)/\alpha$ for some positive constant $\alpha$.
\end{lemma}

\begin{proof}

Noting that the probability density of $u^N(t)$ considered in Lemma \ref{largest eigenvalue bound} is  \eqref{density of Wishart eigenvalue} for all $t$, we can obtain the following tail probability estimation with $\alpha$ being a positive constant independent of $N$,
\begin{align}\label{eq-3.8}
	\mathbb{P} \left( u_N^N(t) \ge x \right)
	= \mathbb{P}^N \left( x_N \ge x \right)
	\le \exp(-\alpha N x), \text{ for } t\ge0.
\end{align}
By Lemma \ref{largest eigenvalue bound} and \eqref{eq-3.8}, we have for $t\ge 0$,
\begin{align} \label{Wishart-largest eigen Lp norm estimation}
	\mathbb{E} \left[ \lambda_N^N(t)^k \right]
	&\le (t+a)^k \mathbb{E} \left[ u_N^N(t)^k \right] = k (t+a)^k \int_0^{\infty} x^{k-1} \mathbb{P} \left( u_N^N(t) \ge x \right) dx \nonumber \\
	&\le k (t+a)^k \int_0^{\infty} x^{k-1} \exp(-\alpha N x) dx = \dfrac{\Gamma(k+1)}{(\alpha N)^k} (t+a)^k \nonumber \\
	&\le (t+a)^k,
\end{align}
for $k \in [0,\alpha N]$, where $\Gamma(x)$ is the gamma function.

Now we apply \eqref{pair formula for empirical measure} and \eqref{quadratic variation of martingale} with $f(x) = x^{n+2}$ for $n \ge -1$ to obtain
\begin{align} \label{Wishart-pair formula of poly for empirical measure}
	\langle x^{n+2}, L_N(t) \rangle
	&= \langle x^{n+2}, L_N(0) \rangle + M_{x^{n+2}}^N(t) + \dfrac{(P+n+1)(n+2)}{N} \int_{0}^t \langle x^{n+1}, L_N(s) \rangle ds \nonumber \\
	&\quad + \dfrac{n+2}{2} \int_{0}^t \iint \sum_{k=0}^n x^k y^{n-k} (x+y) L_N(s)(dx) L_N(s)(dy) ds.
\end{align}
where the martingale term $M_{x^{n+2}}^N(t)$ has the quadratic variation
\begin{align*}
	\langle M_{x^{n+2}}^N \rangle_t
	= \dfrac{4(n+2)^2}{N^2} \int_{0}^t \langle x^{2n+3}, L_N(s) \rangle ds.
\end{align*}
By the Cauchy-Schwarz inequality, Burkholder-Davis-Gundy inequality, H\"{o}lder inequality and the estimation \eqref{Wishart-largest eigen Lp norm estimation}, for $(2n+3)q \le \alpha N$, $q \in \mathbb{N}$, and $\Lambda_q$ being a positive constant depending only on $q$,
\begin{align} \label{Wishart-martingale uniform norm}
	&\mathbb{E} \left[ \left| \sup_{u \in [0,t]} M_{x^{n+2}}^N(u) \right|^q \right]
	\le \left\{ \mathbb{E} \left[ \sup_{u \in [0,t]} M_{x^{n+2}}^N(u)^{2q} \right] \right\}^{1/2} \nonumber \\
	\le& \sqrt{\Lambda_q} \left\{ \mathbb{E} \left[ \langle M_{x^{n+2}}^N \rangle_t^q \right] \right\}^{1/2} \le \dfrac{2^q(n+2)^q\sqrt{\Lambda_q}}{N^q} \left\{ \mathbb{E} \left[ \int_{0}^t \langle x^{2n+3}, L_N(s) \rangle ds \right]^q \right\}^{1/2} \nonumber \\
	= &\dfrac{2^q(n+2)^q\sqrt{\Lambda_q}}{N^q} \left\{ \mathbb{E} \left[ \int_{0}^t \dfrac{1}{N} \sum_{i=1}^N \lambda_i^N(s)^{2n+3} ds \right]^q \right\}^{1/2} \nonumber \\
	\le & \dfrac{2^q(n+2)^q\sqrt{\Lambda_q}}{N^q} \left\{ \mathbb{E} \left[ \int_{0}^t \lambda_N^N(s)^{2n+3} ds \right]^q \right\}^{1/2} \nonumber \\
	\le & \dfrac{2^q(n+2)^q\sqrt{\Lambda_q}}{N^q} \left\{ \mathbb{E} \left[ t^{q-1} \int_{0}^t \lambda_N^N(s)^{(2n+3)q} ds \right] \right\}^{1/2} \nonumber \\
	\le &\dfrac{2^q(n+2)^q\sqrt{\Lambda_q}}{N^q} \left\{ t^{q-1} \int_{0}^t (s+a)^{(2n+3)q} ds \right\}^{1/2} \nonumber \\
	\le &\dfrac{2^q(n+2)^q\sqrt{\Lambda_q t^q (t+a)^{(2n+3)q}}}{N^q}.
\end{align}
 Defining, for $k \ge 1,$
\begin{align*}
	E_t^N(k) = \mathbb{E} \left[ \sup_{u \in [0,t]} \langle x^k, L_N(u) \rangle \right],
\end{align*}
it follows from  \eqref{Wishart-pair formula of poly for empirical measure} that for $n \ge -1$,
\begin{align} \label{Wishart-uniform moment estimation}
	E_t^N(n+2) &\le E_0^N(n+2) + \mathbb{E} \left[ \sup_{u \in [0,t]} M_{x^{n+2}}^N(u) \right] + \dfrac{(P+n+1)(n+2)}{N} \mathbb{E} \left[ \sup_{u \in [0,t]} \int_{0}^u \langle x^{n+1}, L_N(s) \rangle ds \right] \nonumber \\
	&\quad + \dfrac{n+2}{2} \mathbb{E} \left[ \sup_{u \in [0,t]} \int_{0}^u \iint \sum_{k=0}^n x^k y^{n-k} (x+y) L_N(s)(dx) L_N(s)(dy) ds  \right].
\end{align}
For the third and the fourth terms on the right-hand side of \eqref{Wishart-uniform moment estimation},  we have by \eqref{Wishart-largest eigen Lp norm estimation},
\begin{align*}
	&\dfrac{(P+n+1)(n+2)}{N} \mathbb{E} \left[ \sup_{u \in [0,t]} \int_{0}^u \langle x^{n+1}, L_N(s) \rangle ds \right]\\
		\le& \dfrac{(P+n+1)(n+2)}{N} \mathbb{E} \left[ \int_{0}^t |\lambda_N^N(s)|^{n+1} ds \right] \\
	\le &\dfrac{(P+n+1)(n+2)}{N} \int_0^t (s+a)^{n+1} ds \\
	\le& \dfrac{(P+n+1)(n+2)t(t+a)^{n+1}}{N},
\end{align*}
and 
\begin{align*}
	&\quad \dfrac{n+2}{2} \mathbb{E} \left[ \sup_{u \in [0,t]} \int_{0}^u \iint \sum_{k=0}^n x^k y^{n-k} (x+y) L_N(s)(dx) L_N(s)(dy) ds  \right] \\
	&= \dfrac{n+2}{2} \sum_{k=0}^n \mathbb{E} \left[ \sup_{u \in [0,t]} \int_{0}^u \langle x^{k+1}, L_N(s) \rangle \langle y^{n-k}, L_N(s) \rangle + \langle x^k, L_N(s) \rangle \langle y^{n+1-k}, L_N(s) \rangle ds  \right] \\
	&\le \dfrac{(n+2)}{2} \sum_{k=0}^n \mathbb{E} \left[ \sup_{u \in [0,t]} \int_{0}^u |\lambda_N^N(s)|^{k+1} |\lambda_N^N(s)|^{n-k} + |\lambda_N^N(s)|^k |\lambda_N^N(s)|^{n+1-k} ds  \right] \\
	&\le (n+2) (n+1) \mathbb{E} \left[ \int_0^t |\lambda_N^N(s)|^{n+1} ds \right] \\
	&\le (n+2) (n+1) t(t+a)^{n+1}
\end{align*}
for $n+1\le \alpha N$. Hence, by \eqref{Wishart-martingale uniform norm}, \eqref{Wishart-uniform moment estimation}, and the above two estimations, for $n\ge -1$ such that $2n+3 \le \alpha N$ and $t \in [0,T]$, we have
\begin{align*}
	E_t^N(n+2) &\le E_0^N(n+2) + \dfrac{2(n+2)\sqrt{\Lambda_1 t (t+a)^{2n+3}}}{N} \\
	&\quad+ \dfrac{(P+n+1)(n+2)t(t+a)^{n+1}}{N} + (n+2) (n+1) t(t+a)^{n+1}.
\end{align*}
Thus, for all $-1 \le n \le \frac{\alpha N-3}{2}$,   noting that $E_0^N(n+2)\le \E[\lambda_N^N(0)^{n+2}]\le a^{n+2}$ by \eqref{Wishart-largest eigen Lp norm estimation}, we have
\begin{align*}
	E_T^N(n+2) \le C_{a,T}^{n+2},
\end{align*}
for some positive constant $C_{a,T}$ depending on $(a,T)$ only.

 The proof is concluded.
\end{proof}

Now we are ready to prove the following CLT for the eigenvalues of the scaled Wishart process $X^N_t = \tilde{B}^\intercal(t) \tilde{B}(t) / N$, where $\tilde{B}(t)$ is a $P \times N$ Brownian matrix with $P > N-1$.  Noting that under the conditions in Lemma \ref{largest eigenvalue bound},  Lemma \ref{largest eigen uniform bound} implies $\limsup_{N\to \infty}\sup_N \lambda_N^N(0)<\infty$ almost surely. One can check that the conditions (A) - (D) in \cite{Song2019} are satisfied, hence  $\{L_N(t), t\in[0,T]\}_{n\in \mathbb N}$ is tight (see also \cite[Remark 3.3]{Song2019}), and we  know that it converges to $\{\mu_t, t\in[0,T]\}$,  where $\mu_t$ is a scaled Marchenko-Pastur law. Recall that $c=\lim\limits_{N\to \infty} P/N$ and that $\mathcal L_t^N(f)$ is defined by \eqref{eq-functional} in Theorem \ref{Thm-Wishart}.

\begin{theorem} \label{CLT for Wishart}
Assume that $\lim_{N \rightarrow \infty} |P-cN| = 0$, and that for any polynomial $f(x) \in \mathbb{R}[x]$, the initial value $\mathcal{L}_0^N(f)$ converges in probability to a random variable $\L_0(f)$.  Besides, assume the same condition on $\{\lambda_i^N(0), i=1, 2, \dots, N\}$ as in Lemma \ref{largest eigenvalue bound} for all $N\in \mathbb N$. Furthermore, assume that for all $n\in \mathbb N$,
\begin{align}\label{initial-bound}
	\sup_N\E[|N( \langle x^n, L_N(0) \rangle - \langle x^n, \mu_0 \rangle)|^q]<\infty,
\end{align}
for all $q\ge 1$.  Then for any $0 < T < \infty$, there exists a family of processes $\{\L_t(f), t\in [0,T]\}_{f \in \mathbb{R}[x]}$, such that for any $n \in \mathbb{N}$ and any polynomials $P_1, \ldots, P_n \in \mathbb{R}[x]$, the vector-valued process $(\mathcal{L}_t^N(P_1), \ldots, \mathcal{L}_t^N(P_n))_{t \in [0,T]}$ converges to $(\L_t(P_1), \ldots, \L_t(P_n))_{t \in [0,T]}$ in distribution, as $N\to \infty$.

The limit process $\{\L_t(f), t\in [0,T]\}_{f \in \mathbb{R}[x]}$ is characterized by the following properties.
\begin{enumerate}
	\item For $P_1, P_2 \in \mathbb{R}[x]$, $\alpha_1, \alpha_2 \in \mathbb{R}$, $t \in [0,T]$,
	\begin{align*}
		\L_t(\alpha_1 P_1 + \alpha_2 P_2) = \alpha_1 \L_t(P_1) + \alpha_2 \L_t(P_2).
	\end{align*}
	\item The basis $\{\L_t(x^n), t\in [0,T]\}_{n \in \mathbb{N}}$ of $\{\L_t(P), t\in [0,T]\}_{P \in \mathbb{R}[x]}$ satisfies
	\begin{align*}
		\L_t(1) = 0, \quad \L_t(x) = \L_0(x) + G_t(x),
	\end{align*}
	and for $n \ge 0$,
	\begin{align} \label{Wishart-limit equation}
		\L_t(x^{n+2})
		&= \L_0(x^{n+2}) + c(n+2) \int_{0}^t \L_s(x^{n+1}) ds + (n+2)(n+1) \int_{0}^t \langle x^{n+1}, \mu_s \rangle ds \nonumber \\
		&+ (n+2) \sum_{k=0}^n \int_0^t \L_s (x^{n-k}) \mu_s(x^{k+1}) + \L_s (x^{n+1-k}) \mu_s(x^k) ds + G_t(x^{n+2}),
	\end{align}
	where $\{G_t(x^n), t \in [0,T]\}_{n \in \mathbb{N}}$ is a family of  centered Gaussian processes with covariance
	\begin{align}\label{Wishart-Gaussian covariance}
		\mathbb{E} \left[ G_t(x^n) G_s(x^m) \right]
		= 4mn \int_0^{t \wedge s} \langle x^{n+m-1}, \mu_u \rangle du, \quad n,m \ge 1.
	\end{align}
\end{enumerate}
\end{theorem}
\begin{proof}

 First, note that by  Lemma \ref{largest eigen uniform bound} and  Corollary \ref{Coro-Lp},  $Q_t^N(x^n)$ defined by \eqref{eq-Q} converges in distribution to a centered Gaussian family $\{G_t(x^n), t\in[0,T]\}_{n\in \mathbb N}$ with covariance given by \eqref{Wishart-Gaussian covariance}.
Furthermore, by  \eqref{eq-functional}, \eqref{eq-Q} and \eqref{eq-3.6}, for $n \ge -1$, we have 
\begin{align}
	Q_t^N(x^{n+2}) &= \mathcal{L}_t^N(x^{n+2}) - \mathcal{L}_0^N(x^{n+2}) - c(n+2) \int_{0}^t \mathcal{L}_s^N(x^{n+1}) ds - (n+2)(n+1) \int_{0}^t \langle x^{n+1}, \mu_s \rangle ds\notag \\
	&\quad- (n+2) \int_0^t \mathcal{L}_s^N\left( \int \sum_{k=0}^n x^k y^{n-k} (x+y) \mu_s(dx) \right) ds\notag \\
	&\quad- \dfrac{N(n+2)}{2} \int_{0}^t \iint \sum_{k=0}^n x^k y^{n-k} (x+y) [L_N(s)(dx) - \mu_s(dx)] [L_N(s)(dy) - \mu_s(dy)] ds \notag\\
	&= \mathcal{L}_t^N(x^{n+2}) - \mathcal{L}_0^N(x^{n+2}) - c(n+2) \int_{0}^t \mathcal{L}_s^N(x^{n+1}) ds - (n+2)(n+1) \int_{0}^t \langle x^{n+1}, \mu_s \rangle ds \notag\\
	&\quad- (n+2) \sum_{k=0}^n \int_0^t \mathcal{L}_s^N (x^{n-k}) \mu_s(x^{k+1}) + \mathcal{L}_s^N (x^{n+1-k}) \mu_s(x^k) ds \notag\\
	&\quad- \dfrac{(n+2)}{2N} \sum_{k=0}^n \int_{0}^t \mathcal{L}_s^N (x^{n-k}) \mathcal{L}_s^N (x^{k+1}) + \mathcal{L}_s^N (x^{n+1-k}) \mathcal{L}_s^N (x^k) ds. \label{eq-3.16}
\end{align}

 In Corollary \ref{Thm for bounded general Wishart process} and Corollary \ref{Coro-Lp}, we have shown $Q_t^N(x^{n+2})-NM_{x^{n+2}}^N$ converges to $0$ almost surely and in $L^q$ for all $q \ge 1$ as $N \rightarrow \infty$, uniformly in $t \in [0,T]$. Thus, by \eqref{Wishart-martingale uniform norm}, \eqref{eq-3.16}, and the condition \eqref{initial-bound}, it is not difficult to show
\begin{align*}
	\sup_{N \in \mathbb{N}} \mathbb{E} \left[ \sup_{t \in [0,T]} \left| \mathcal{L}_t^N(x^n) \right|^q \right] < \infty,
\end{align*}
for $q \ge 1$ and $n\in \mathbb N$ by using an induction argument on $n$.

To estimate the last term on the right-hand side of \eqref{eq-3.16},   we apply the Cauchy-Schwarz inequality to obtain, for $0 \le k \le n$,
\begin{align*}
	&\quad \mathbb{E} \left[ \sup_{t \in [0,T]} \left| \dfrac{n+2}{2N} \int_0^t \mathcal{L}_s^N (x^{n-k}) \mathcal{L}_s^N (x^{k+1}) ds \right|^q \right] \\
	&\le \dfrac{(n+2)^q T^q}{2^q N^q} \mathbb{E} \left[ \sup_{t \in [0,T]} \left| \mathcal{L}_t^N (x^{n-k}) \right|^q \sup_{t \in [0,T]} \left| \mathcal{L}_t^N (x^{k+1}) \right|^q \right] \\
	&\le \dfrac{(n+2)^q T^q}{2^q N^q} \left\{ \mathbb{E} \left[ \sup_{t \in [0,T]} \left| \mathcal{L}_t^N (x^{n-k}) \right|^{2q}  \right] \mathbb{E} \left[ \sup_{t \in [0,T]} \left| \mathcal{L}_t^N (x^{k+1}) \right|^{2q} \right] \right\}^{1/2} \\
	&\le C(n,T,q) N^{-q},
\end{align*}
for some constant $C(n,T,q)$. Thus, the last term on the right-hand side of \eqref{eq-3.16} converges to 0 
in $L^q$ for $q > 1$, as $N$ tends to infinity. By Markov inequality and Borel-Cantelli Lemma, one can also obtain the almost sure convergence. 

If we define
\begin{align}\label{eq-37}
	\tilde{Q}_t^N(x^{n+2})
	&= \mathcal{L}_t^N(x^{n+2}) - \mathcal{L}_0^N(x^{n+2}) - c(n+2) \int_{0}^t \mathcal{L}_s^N(x^{n+1}) ds -  (n+2)(n+1)\int_{0}^t \langle x^{n+1}, \mu_s \rangle ds\notag \\
	&\quad- (n+2) \sum_{k=0}^n \int_0^t \mathcal{L}_s^N (x^{n-k}) \mu_s(x^{k+1}) + \mathcal{L}_s^N (x^{n+1-k}) \mu_s(x^k) ds,
\end{align}
for $n \ge -1$, then the difference $|\tilde{Q}_t^N(x^{n+2}) - Q_t^N(x^{n+2})|$ converges to $0$ almost surely and in $L^q$ for $q > 1$. Thus, Corollary \ref{Thm for bounded general Wishart process} implies that $(\tilde{Q}_t^N(x^k), \tilde{Q}_t^N(x^{k-1}), \ldots, \tilde{Q}_t^N(x))_{t \in [0,T]}$ converges in distribution to $(G_t(x^k), G_t(x^{k-1}), \ldots, G_t(x))_{t \in [0,T]}$ with covariance \eqref{Wishart-Gaussian covariance}.

Now we deduce the convergence in distribution of $(\mathcal{L}_t^N(x^k))_{t \in [0,T]}$ for $k \in \mathbb{N}$. First of all, we have $\mathcal{L}_t^N(1) = 0$ and $\mathcal{L}_t^N(x) = \mathcal{L}_0^N(x) + \tilde{Q}_t^N(x)$ converges in distribution since the initial value converges in probability. By induction,  if we assume $(\mathcal{L}_t^N(x^k), \ldots, \mathcal{L}_t^N(x))_{t \in [0,T]}$ convergence in distribution to $(\L_t(x^k), \ldots, \L_t(x))_{t \in [0,T]}$, then the  convergence in distribution of $(\tilde{Q}_t^N(x^{k+1}), \tilde{Q}_t^N(x^k), \ldots, \tilde{Q}_t^N(x))_{t \in [0,T]}$  implies that $(\tilde{Q}_t^N(x^{k+1}), \mathcal{L}_t^N(x^k), \ldots, \mathcal{L}_t^N(x))_{t \in [0,T]}$ converges in distribution, and hence $(\mathcal{L}_t^N(x^{k+1}), \ldots, \mathcal{L}_t^N(x))_{t \in [0,T]}$ converges in distribution.

Thus, by \eqref{eq-37} we have
\begin{align*}
	G_t(x^{n+2})
	&\overset{d}{=} \mathcal{L}_t(x^{n+2}) - \mathcal{L}_0(x^{n+2}) - c(n+2) \int_{0}^t \mathcal{L}_s(x^{n+1}) ds - (n+2)(n+1) \int_{0}^t \langle x^{n+1}, \mu_s \rangle ds \\
	&\quad- (n+2) \sum_{k=0}^n \int_0^t \mathcal{L}_s (x^{n-k}) \mu_s(x^{k+1}) + \mathcal{L}_s (x^{n+1-k}) \mu_s(x^k) ds,
\end{align*}
where ``$\overset d=$'' means equality in distribution. The proof is concluded.
\end{proof}

\begin{remark} By the self-similarity of Brownian motion, when $X_0^N = 0$, we have $X_t^N \overset{d}{=} t X_1^N$. Thus, $(\lambda_1^N(t), \ldots, \lambda_N^N(t))\overset d=(t \lambda_1^N(1), \ldots, t \lambda_N^N(1))$. Therefore,
\begin{align*}
	\langle f(x), L_N(t) \rangle = \dfrac{1}{N} \sum_{i=1}^N f(\lambda_i^N(t)) \overset d= \dfrac{1}{N} \sum_{i=1}^N f(t \lambda_i^N(1)) = \langle f(tx), L_N(1) \rangle,
\end{align*}
 and
\begin{align*}
	\langle f(x), \mu_t \rangle \overset d= \langle f(tx), \mu_1 \rangle.
\end{align*}
Hence, $\mathcal{L}_t^N(f(x)) \overset d= \mathcal{L}_1^N(f(tx))$, and thus, $\L_t(f(x)) \overset d= \L_1(f(tx))$. With these identities and the linearity of $\L_t(\cdot)$, \eqref{Wishart-limit equation} can be simplified as, for $n\ge 0$,
\begin{align}
	\L_1(x^{n+2}) &= c \L_1(x^{n+1}) + (n+1) \langle x^{n+1}, \mu_1 \rangle + \sum_{k=0}^n \left( \L_1(x^{n-k}) \langle x^{k+1}, \mu_1 \rangle + \L_1(x^{n+1-k}) \langle x^{k}, \mu_1 \rangle \right) \notag\\
	&\quad +  \dfrac{1}{t^{n+2}} G_t(x^{n+2}), \quad t > 0,\label{eq-3.19}
\end{align}
where the Gaussian family $\{G_t(x^n), t\in[0,T]\}_{n\in \mathbb N}$ has the covariance functions
\begin{align*}
	\mathbb{E} \left[ G_t(x^n) G_s(x^m) \right]
	=  \dfrac{4mn}{m+n} (t\wedge s)^{n+m}\langle x^{n+m-1}, \mu_1 \rangle, \quad n,m \ge 1.
\end{align*}
Note that the case $t=1$ corresponds to the classical Wishart matrix, and $\mu_1$ is the Marchenko–Pastur law. More precisely, recalling that $\L_1(1) = 0$ and  $\L_1(x) = G_1(x) $,  we get by \eqref{eq-3.19} $\L_1(x^2) = \langle x, \mu_1 \rangle + (c+1) G_1(x) + G_1(x^2)$, for $m\ge 3$, and more generally $\L_1(x^m) = c_{m,0} + c_{m,1} G_1(x) + \ldots + c_{m,m} G_1(x^m)$ for some coefficients $(c_{m,j})_{0 \le j \le m}$ which are determined recursively by \eqref{eq-3.19}.
\end{remark}

We now study a more general particle systems:
\begin{align} \label{Wishart SDE with drift}
	d \lambda_i^N(t)
	= 2 \dfrac{\sqrt{\lambda_i^N(t)}}{\sqrt{N}} dW_i(t) + \left( b_N(\lambda_i^N(t)) + \dfrac{1}{N} \sum_{j:j \neq i} \dfrac{\lambda_i^N(t) + \lambda_j^N(t)}{\lambda_i^N(t) - \lambda_j^N(t)} \right) dt, ~ 1\le i \le N, ~ t\ge 0.
\end{align}
Compared to \eqref{Wishart eigenvalue SDE}, the constant $P/N$ is replaced by a function $b_N(x)$ that will be assumed to converge to a constant $c$ in Corollary \ref{general wishart} below. Despite  the extension being small, the system \eqref{Wishart SDE with drift} may not correspond to eigenvalues of a matrix SDE, and may not have an explicit joint density function or stationary distribution, and hence cannot be treated in the same way as for the eigenvalues of Wishart process.

\begin{corollary} \label{general wishart}
Consider the SDEs \eqref{Wishart SDE with drift}, where $b_N(x)$ satisfies, for some constant  $c \ge 1$,
\begin{align} \label{condition on b_N}
	\lim_{N \rightarrow \infty} N \| b_N(x) - c \|_{L^{\infty}(\mathbb{R})} = 0.
\end{align}
Assume the same initial conditions as  in Theorem \ref{CLT for Wishart}. Then the conclusion of Theorem \ref{CLT for Wishart} still holds.
\end{corollary}

\begin{proof}
Let $p_1 = N (c - \|b_N(x) - c\|_{L^{\infty}})$ and $p_2 = N (c + \|b_N(x) - c\|_{L^{\infty}})$ be two constants depending on $N$.  Then \eqref{condition on b_N} implies  $p_2 \ge p_1 > N-1$ when $N$ is large. Clearly,  $p_1 \le N \|b_N(x)\|_{L^\infty(\R)} \le p_2$. Consider the following two systems of SDEs:
\begin{align} \label{eq-3.22'}
	d x_i^N(t)
	= 2 \dfrac{\sqrt{x_i^N(t)}}{\sqrt{N}} dW_i(t) + \left( \dfrac{p_1}{N} + \dfrac{1}{N} \sum_{j:j \neq i} \dfrac{x_i^N(t) + x_j^N(t)}{x_i^N(t) - x_j^N(t)} \right) dt, ~ 1\le i \le N, ~ t\ge 0,
\end{align}
and
\begin{align}\label{eq-3.23'}
	d y_i^N(t)
	= 2 \dfrac{\sqrt{y_i^N(t)}}{\sqrt{N}} dW_i(t) + \left( \dfrac{p_2}{N} + \dfrac{1}{N} \sum_{j:j \neq i} \dfrac{y_i^N(t) + y_j^N(t)}{y_i^N(t) - y_j^N(t)} \right) dt, ~ 1\le i \le N, ~ t\ge 0,
\end{align}
with the initial conditions $x_i^N(0) = y_i^N(0) = \lambda_i^N(0)$. By the comparison principle in Corollary \ref{Thm-compare of eigenvalue}, we have
\begin{align*}
	\mathbb{P} (x_i^N(t) \le \lambda_i^N(t) \le y_i^N(t), ~ \forall t \ge 0, ~ \forall 1 \le i \le N) = 1.
\end{align*}
Thus, almost surely, 
\begin{align}
	\sup_{t \in [0,T]} \langle |x|^p, L_N(t) \rangle
	&= \sup_{t \in [0,T]} \dfrac{1}{N} \sum_{i=1}^N |\lambda_i^N(t)|^p \notag\\
	&\le \sup_{t \in [0,T]} \dfrac{1}{N} \sum_{i=1}^N (|x_i^N(t)|^p + |y_i^N(t)|^p)\notag \\
	&\le \sup_{t \in [0,T]} \langle |x|^p, L_N^{(x)}(t) \rangle + \sup_{t \in [0,T]} \langle |x|^p, L_N^{(y)}(t) \rangle,\label{eq3.24'}
\end{align}
where $\{L_N^{(x)}(t), t \in [0,T]\}_{N \in \mathbb{N}}$ and $\{L_N^{(y)}(t), t \in [0,T]\}_{N \in \mathbb{N}}$ are the empirical measures of the two particle systems $(x_i^N(t))_{1 \le i \le N}$ and $(y_i^N(t))_{1 \le i \le N}$, respectively. 

Noting that $p_1/N$ and $p_2/N$ converge to $c$  as $N \rightarrow \infty$ by \eqref{condition on b_N}, we have that Lemma \ref{largest eigen uniform bound} holds for the two systems \eqref{eq-3.22'} and \eqref{eq-3.23'}, and thus also holds for \eqref{Wishart SDE with drift} by \eqref{eq3.24'}. Furthermore, condition \eqref{condition on b_N} also yields that $b_N(x)\to c$ uniformly as $N\to \infty$, and hence  \eqref{eq-3.16} still holds. Then the rest of the proof  follows that of Theorem \ref{CLT for Wishart}.
\end{proof}

\subsection{Application to Dyson's Brownian motion} \label{sec-dyson}

In this subsection, we discuss the CLT for the Dyson's Brownian motion. It was shown in \cite{Anderson2010,Graczyk2014,Song2019}, the scaled symmetric Brownian motion $X^N_t = (\tilde{B}^\intercal(t) + \tilde{B}(t)) / \sqrt{2N}$, where $\tilde{B}(t)$ is a $N \times N$ Brownian matrix, is the solution of the matrix SDE \eqref{matrix SDE} with the coefficient functions
\begin{align*}
	g_N(x) h_N(y) = \dfrac{1}{\sqrt{2N}}, \quad
	b_N(x) = 0.
\end{align*}
The system of SDEs of the eigenvalue processes, that is, the Dyson's Brownian motion, is  
\begin{align} \label{Dyson eigenvalue SDE}
	d \lambda_i^N(t)
	= \dfrac{\sqrt{2}}{\sqrt{N}} dW_i(t) + \dfrac{1}{N} \sum_{j:j \neq i} \dfrac{1}{\lambda_i^N(t) - \lambda_j^N(t)} dt, ~ 1\le i \le N, ~ t\ge 0.
\end{align}
In this case, we have
\begin{align}\label{eq-3.21}
	NG_N(x,y) = G(x,y) = 1, \quad
	b(x) = 0.
\end{align}

Here, we consider the distribution $\mathbb{Q}^N$ on $\Delta_N' = \{x=(x_1, x_2, \dots, x_N) \in \mathbb{R}^N: x_1 < \ldots < x_N \}$ with the density function
\begin{align} \label{density of Dyson eigenvalue}
	C_N \exp \left( -\dfrac{N}{4} \sum_{i=1}^N x_i^2 \right) \prod_{1\le j<i\le N} |x_i - x_j|,
\end{align}
where $C_N$ is a normalization constant.

Similar to the Wishart process, we can obtain the following central limit theorem.

\begin{theorem} \label{CLT for Dyson}
Let  $\xi^N=(\xi^N_1, \dots, 
\xi^N_N)$ be a random vector that is independent of $(W_1, \dots, W_N)$ and has  \eqref{density of Dyson eigenvalue} as its joint probability density function.
Assume that $(\lambda_1^N(0), \dots, \lambda_N^N(0))$ is  independent of $(W_1, \dots, W_N)$  and that there exist constants $a , b \ge 0$, such that
\begin{align}\label{eq-bound-initial}
	\sqrt{a} \xi_i^N - b \le \lambda_i^N(0) \le \sqrt{a} \xi_i^N + b
\end{align}
for $1 \le i \le N$ almost surely. Besides, assume that for any polynomial $f(x) \in \mathbb{R}[x]$, the initial value $\mathcal{L}_0^N(f)$ converges in probability to a random variable $\L_0(f)$. Furthermore, assume that for all  $n \in \mathbb N$,
\begin{align*}
	\sup_N\E[|N(\langle x^n, L_N(0)\rangle -\langle x^n, \mu_0\rangle)|^p]<\infty,
\end{align*}
for all $p\ge 1$.

Then for any $0 < T < \infty$, there exists a family of processes $\{\L_t(f), t\in [0,T]\}_{f \in \mathbb{R}[x]}$, such that for any $n \in \mathbb{N}$ and any polynomial $P_1, \ldots, P_n \in \mathbb{R}[x]$, the vector-valued process $(\mathcal{L}_t^N(P_1), \ldots, \mathcal{L}_t^N(P_n))_{t \in [0,T]}$ converges to $(\L_t(P_1), \ldots, \L_t(P_n))_{t \in [0,T]}$ in distribution.

The limit process $\{\L_t(f), t\in [0,T]\}_{f \in \mathbb{R}[x]}$ is characterized by the following properties.
\begin{enumerate}
	\item For $P_1, P_2 \in \mathbb{R}[x]$, $\alpha_1, \alpha_2 \in \mathbb{R}$, $t \in [0,T]$,
	\begin{align*}
		\L_t(\alpha_1 P_1 + \alpha_2 P_2) = \alpha_1 \L_t(P_1) + \alpha_2 \L_t(P_2).
	\end{align*}
	\item The basis $\{\L_t(x^n), t\in [0,T]\}_{n \in \mathbb{N}}$ of $\{\L_t(f), t\in [0,T]\}_{f \in \mathbb{R}[x]}$ satisfies
	\begin{align*}
		\L_t(1) = 0, \quad \L_t(x) = \L_0(x) + G_t(x),
	\end{align*}
	and for $n \ge 0$,
	\begin{align} \label{Dyson limit equation}
		\L_t(x^{n+2})
		= &\L_0(x^{n+2}) + \dfrac{(n+2)(n+1)}{2} \int_{0}^t \langle x^n, \mu_s \rangle ds\notag\\
		& + (n+2) \sum_{k=0}^n \int_0^t \L_s (x^{n-k}) \mu_s(x^k) ds + G_t(x^{n+2}),
	\end{align}
	where $\{G_t(x^n), t \in [0,T]\}_{n \in \mathbb{N}}$ is a centered Gaussian family with the covariance
	\begin{align*}
		\mathbb{E} \left[ G_t(x^n) G_s(x^m) \right]
		= 2mn \int_0^{t \wedge s} \langle x^{n+m-2}, \mu_u \rangle du, \quad n,m \ge 1.
	\end{align*}
\end{enumerate}
\end{theorem}

\begin{proof}
The proof is similar to the proofs of the Wishart case (Lemma \ref{largest eigenvalue bound},  Lemma \ref{largest eigen uniform bound} and Theorem \ref{CLT for Wishart}), which is sketched below.  

Consider the following SDE, for $1\le i \le N$,
\begin{align*}
	d u_i^N(t)
	= \dfrac{\sqrt{2}}{\sqrt{N(t+a)}} dW_i(t) + \dfrac{1}{t+a} \left( - \dfrac{1}{2} u_i^N(t) + \dfrac{1}{N} \sum_{j:j \neq i} \dfrac{1}{u_i^N(t) - u_j^N(t)} \right) dt, \quad t \ge 0.
\end{align*}
Then $\frac{d}{dt} \mathbb{E} [f(u^N(t))]$ vanishes for any $f \in C_b^2(\mathbb{R})$ if $u^N(t)$ has the distribution $\mathbb{Q}^N$ given in \eqref{density of Dyson eigenvalue}, and hence the process $u^N(t)$ with initial value $u^N(0) = \xi^N$ is stationary (see \cite[Lemma 4.3.17 ]{Anderson2010}). Let $v_i^N(t) = \sqrt{t+a} u_i^N(t) + b$ for $1 \le i \le N$. Then $v^N(t)$ and $\lambda^N(t)$ solve the same SDEs \eqref{Dyson eigenvalue SDE}, and by the comparison principle in Corollary \ref{Thm-compare of eigenvalue}, we have
\begin{align*}
	\lambda_i^N(t) \le v_i^N(t)
	= \sqrt{t+a} u_i^N(t) + b.
\end{align*}
A similar argument leads to
\begin{align*}
	- \lambda_i^N(t) \le - \sqrt{t+a} u_i^N(t) + b.
\end{align*}
Therefore, 
\begin{align*}
	|\lambda_i^N(t)| \le \sqrt{t+a} |u_i^N(t)| + b.
\end{align*}

Using the tail probability estimation based on the density function \eqref{density of Dyson eigenvalue} of $u_i^N(t)$, 
\begin{align*}
	\mathbb{P} \left( |u_i^N(t)| \ge x \right)
	\le \mathbb{P} (u_N^N(t) \ge x) + \mathbb{P} (u_1^N(t) \le -x)
	\le 2 \mathbb{P} (u_N^N(t) \ge x)
	\le \exp(-\alpha N x),
\end{align*}
where  $\alpha$ is positive constant  independent of $N$,
we obtain 
\begin{align*}
	\mathbb{E} \left[ |\lambda_i^N(t)|^k \right]
	&\le \mathbb{E} \left[ \left( \sqrt{t+a} |u_i^N(t)| + b \right)^k \right] \\
	&\le 2^k \sqrt{t+a}^k \mathbb{E} \left[ |u_i^N(t)|^k \right] + 2^k b^k \\
	&= 2^k \sqrt{t+a}^k k \int_0^{\infty} x^{k-1} \mathbb{P} \left( |u_i^N(t)| \ge x \right) dx + 2^k b^k \\
	&\le 2^k \sqrt{t+a}^k k \int_0^{\infty} x^{k-1} \exp(-\alpha N x) dx + 2^k b^k \\
	&= 2^k \sqrt{t+a}^k \dfrac{\Gamma(k+1)}{(\alpha N)^k} + 2^k b^k \\
	&\le 2^k \sqrt{t+a}^k + 2^k b^k \\
	&\le 2 \left( 4b^2 + 4(t+a) \right)^{k/2}
\end{align*}
for $k \in [0, \alpha N]$. Then a similar argument in the proof of Lemma \ref{largest eigen uniform bound} leads to
\begin{align}\label{eq3.29}
	\mathbb{E} \left[ \sup_{t \in [0,T]} \langle |x|^p, L_N(t) \rangle \right] \le C(a,b,T)^p
\end{align}
for some positive constant $C(a,b,T)$ depending only on $(a,b,T)$ and all $p \ge 0$, $N \ge \alpha p$ for some positive constant $\alpha$.

Then applying Corollary \ref{Coro-Lp} and following the approach in the proof of Theorem \ref{CLT for Wishart}, we may get the desired result. 
\end{proof}

\begin{remark}
The above result was obtained in \cite[Theorem 4.3.20]{Anderson2010}, under a slightly stronger condition  on the initial value. We would like to point out that there should be a constant factor $2/\beta$ in the covariance function which equals to $2$ in the real case and equals to $1$ in the complex case in \cite{Anderson2010}.
\end{remark}

Similar to the Wishart case, the self-similarity of the Brownian motion implies $\L_t(f(x)) \overset d= \L_1(f(\sqrt{t}x))$ and $\langle f(x), \mu_t \rangle = \langle f(\sqrt{t}x), \mu_1 \rangle$ when the initial value $X_0^N = 0$. Thus, \eqref{Dyson limit equation} can be simplified as, for $n\ge 0$,
\begin{align}\label{eq-3.24}
	\L_1(x^{n+2})
	= (n+1) \langle x^n, \mu_1 \rangle + 2 \sum_{k=0}^n \L_1 (x^{n-k}) \mu_1(x^k) + \dfrac{1}{t^{\frac{n+2}2}} G_t(x^{n+2}), \quad t > 0,
\end{align}
with covariance functions
\begin{align*}
	\mathbb{E} \left[ G_t(x^n) G_s(x^m) \right]
	=  \dfrac{4mn}{m+n} (t\wedge s)^{\frac{m+n}2} \langle x^{m+n-2}, \mu_1 \rangle, \quad n,m \ge 1.
\end{align*}
The case $t=1$ corresponds to the classical GOE matrix, and $\mu_1$ is the semicircle law.  Some beginning terms are $\L_1(1)=0, \L_1(x) = G_1(x)$ and $\L_1(x^2) = 1 + G_1(x^2)$. By  \eqref{eq-3.24},  for $m\ge 2$,  $\L_1(x^m)$ has the distribution of a linear combination of central Gaussian variables $\{G_1(x^j),1 \le j \le m\}$.

The following Corollary extends the result of Theorem \ref{CLT for Dyson}.

\begin{corollary} \label{general Dyson}
Consider the following SDEs
\begin{align} \label{Dyson SDE with drift}
	d \lambda_i^N(t)
	= \dfrac{\sqrt{2}}{\sqrt{N}} dW_i(t) + \left( b_N(\lambda_i^N(t)) + \dfrac{1}{N} \sum_{j:j \neq i} \dfrac{1}{\lambda_i^N(t) - \lambda_j^N(t)} \right) dt, ~ 1\le i \le N, ~ t\ge 0,
\end{align}
where  $b_N(x)$ satisfies, for some constant $c\in \R$,
\begin{align} \label{con-gen-Dyson-b_N}
	\lim_{N \rightarrow \infty} N \| b_N(x) - c \|_{L^{\infty}(\mathbb{R})} = 0.
\end{align}
Furthermore, assume the same initial conditions as in Theorem \ref{CLT for Dyson}. Then the conclusion of Theorem \ref{CLT for Dyson} still holds with \eqref{Dyson limit equation} replaced by
\begin{align}
	\L_t(x^{n+2})
	= &\L_0(x^{n+2}) + c(n+2) \int_0^t \L_s(x^{n+1}) ds + \dfrac{(n+2)(n+1)}{2} \int_{0}^t \langle x^n, \mu_s \rangle ds\notag\\
	& + (n+2) \sum_{k=0}^n \int_0^t \L_s (x^{n-k}) \mu_s(x^k) ds + G_t(x^{n+2}),
\end{align}
for $n \ge -1$.
\end{corollary}

\begin{proof}
Set $c_1 = c-1$ and $c_2 = c+1$. Then by \eqref{con-gen-Dyson-b_N}, there exist $N_0 \in \mathbb{N}$ such that for $N \ge N_0$, $c_1 \le \|b_N(x)\|_{L^\infty(\R)} \le c_2$. Without loss of generality, we assume $c_1 \le \|b_N(x)\|_{L^\infty(\R)} \le c_2$ for all $N\ge1$.

Consider the following two systems of SDEs:
\begin{align} \label{eq-3.32}
	d x_i^N(t)
	= \dfrac{\sqrt{2}}{\sqrt{N}} dW_i(t) + \left( c_1 + \dfrac{1}{N} \sum_{j:j \neq i} \dfrac{1}{x_i^N(t) - x_j^N(t)} \right) dt, ~ 1\le i \le N, ~ t\ge 0.
\end{align}
and
\begin{align} \label{eq-3.33}
	d y_i^N(t)
	= \dfrac{\sqrt{2}}{\sqrt{N}} dW_i(t) + \left( c_2 + \dfrac{1}{N} \sum_{j:j \neq i} \dfrac{1}{y_i^N(t) - y_j^N(t)} \right) dt, ~ 1\le i \le N, ~ t\ge 0,
\end{align}
with the initial conditions $x_i^N(0) = y_i^N(0) = \lambda_i^N(0)$ for $1 \le i \le N$. By the comparison principle Theorem \ref{Thm-compare of particle}, we have
\begin{align*}
	\mathbb{P} \left( x_i^N(t) \le \lambda_i^N(t) \le y_i^N(t), \ \forall 1 \le i \le N, \ \forall t > 0 \right) = 1.
\end{align*}
Thus, for $p \ge 1$, we have
\begin{align}
	&\sup_{t \in [0,T]} \langle |x|^p, L_N(t) \rangle
	= \sup_{t \in [0,T]} \dfrac{1}{N} \sum_{i=1}^N |\lambda_i^N(t)|^p \le \sup_{t \in [0,T]} \dfrac{1}{N} \sum_{i=1}^N (|x_i^N(t)|^p + |y_i^N(t)|^p)\notag \\
	&\le \sup_{t \in [0,T]} \dfrac{1}{N} \sum_{i=1}^N 2^p (|x_i^N(t) - c_1t|^p + (c_1t)^p + |y_i^N(t) - c_2t|^p + (c_2t)^p) \notag\\
	&\le 2^p \left(\sup_{t \in [0,T]} \langle |x|^p, L_N^{(x)}(t) \rangle + \sup_{t \in [0,T]} \langle |x|^p, L_N^{(y)}(t) \rangle + (c_1 T)^p +  (c_2 T)^p\right),\label{eq3.36}
\end{align}
almost surely, where $\{L_N^{(x)}(t), t \in [0,T]\}_{N \in \mathbb{N}}$ and $\{L_N^{(y)}(t), t \in [0,T]\}_{N \in \mathbb{N}}$ are the empirical measures of the two particle systems $(x_i^N(t) - c_1t)_{1 \le i \le N}$ and $(y_i^N(t) - c_2t)_{1 \le i \le N}$, respectively.

It is easy to verify that both $(x_i^N(t) - c_1t)_{1 \le i \le N}$ and $(y_i^N(t) - c_2t)_{1 \le i \le N}$ solve the Dyson's SDEs \eqref{Dyson eigenvalue SDE}. By \eqref{eq3.29} in the proof Theorem \ref{CLT for Dyson}, we have
\[
	\mathbb{E} \left[ \sup_{t \in [0,T]} \langle |x|^p, L_N^{(x)}(t) \rangle \right] \le C(a,b,T)^p
~ \text{ and }~
	\mathbb{E} \left[ \sup_{t \in [0,T]} \langle |x|^p, L_N^{(y)}(t) \rangle \right] \le C(a,b,T)^p,
\]
and consequently, by \eqref{eq3.36}
\begin{align*}
	\mathbb{E} \left[ \sup_{t \in [0,T]} \langle |x|^p, L_N(t) \rangle \right] \le C(a,b,T)^p,
\end{align*}
for some positive constant $C(a,b,T)$ depending only on $(a,b,T)$ and all $p \ge 1$, $N \ge \alpha p$ for some positive constant $\alpha$.

Note that \eqref{con-gen-Dyson-b_N} also implies that $b_N(x)$ converges to the constant $c$ uniformly as $N \rightarrow \infty$.  Then applying Corollary \ref{Coro-Lp} and following the approach in the proof of Theorem \ref{CLT for Wishart}, we get the desired result. 

\end{proof}

\subsection{Application to eigenvalues of symmetric OU matrix} \label{sec-ou}

In this subsection, we discuss the CLT for the eigenvalues of a symmetric Ornstein-Uhlenbeck matrix process. It was shown in \cite{Chan1992}, the symmetric $N \times N$ matrix $X^N(t)$, whose entries $\{X_{ij}^N(t), i \le j\}$ are independent Ornstein-Uhlenbeck processes with invariant distribution $N(0, (1 + \delta_{ij})/(2N))$, where $\delta_{ij}$ is the Kronecker delta function,  is the solution of the matrix SDE \eqref{matrix SDE} with the coefficient functions
\begin{align*}
	g_N(x) h_N(y) = \dfrac{1}{2\sqrt{N}}, \quad
	b_N(x) = -\dfrac{1}{2}x.
\end{align*}
The SDEs of the eigenvalue processes are  
\begin{align} \label{OU eigenvalue SDE}
	d \lambda_i^N(t)
	= \dfrac{1}{\sqrt{N}} dW_i(t) + \left( -\dfrac{1}{2} \lambda_i^N(t) + \dfrac{1}{2N} \sum_{j:j \neq i} \dfrac{1}{\lambda_i^N(t) - \lambda_j^N(t)} \right) dt, ~ 1\le i \le N, ~ t\ge 0.
\end{align}
In this case, we have
\begin{align*}
	NG_N(x,y) = G(x,y) = \dfrac{1}{2}, \quad
	b(x) = -\dfrac{1}{2}x.
\end{align*}

Similar to the eigenvalues of Wishart process and Dyson's Brownian motion, we have the following CLT.
\begin{theorem} \label{CLT for OU}

Let  $\xi^N=(\xi^N_1, \dots, \xi^N_N)$ be a random vector that is independent of $(W_1, \dots, W_N)$  and has \eqref{density of Dyson eigenvalue} as its joint probability density function. Assume that $(\lambda_1^N(0), \dots, \lambda_N^N(0))$ is independent of $(W_1, \dots, W_N)$ and that  there exist constants $a , b \ge 0$, such that
\begin{align*}
	\sqrt{a} \xi_i^N - b \le \lambda_i^N(0) \le \sqrt{a} \xi_i^N + b
\end{align*}
for $1 \le i \le N$ almost surely.

Besides, assume that for any polynomial $f(x) \in \mathbb{R}[x]$, the initial value $\mathcal{L}_0^N(f)$ converges in probability to a random variable $\L_0(f)$. Furthermore, assume that for all $n\in \mathbb N$,
\begin{align*}
	\sup_N\E[|N(\langle x^n, L_N(0)\rangle -\langle x^n, \mu_0\rangle)|^p]<\infty,
\end{align*}
for all $p\ge 1$.
	
Then for any $0 < T < \infty$, there exists a family of processes $\{\L_t(f), t\in [0,T]\}_{f \in \mathbb{R}[x]}$, such that for any $n \in \mathbb{N}$ and any polynomial $P_1, \ldots, P_n \in \mathbb{R}[x]$, the vector-valued process $(\mathcal{L}_t^N(P_1), \ldots, \mathcal{L}_t^N(P_n))_{t \in [0,T]}$ converges to $(\L_t(P_1), \ldots, \L_t(P_n))_{t \in [0,T]}$ in distribution.
	
The limit process $\{\L_t(f), t\in [0,T]\}_{f \in \mathbb{R}[x]}$ is characterized by the following properties.
\begin{enumerate}
	\item For $P_1, P_2 \in \mathbb{R}[x]$, $\alpha_1, \alpha_2 \in \mathbb{R}$, $t \in [0,T]$,
	\begin{align*}
		\L_t(\alpha_1 P_1 + \alpha_2 P_2) = \alpha_1 \L_t(P_1) + \alpha_2 \L_t(P_2).
	\end{align*}
	\item The basis $\{\L_t(x^n), t\in [0,T]\}_{n \in \mathbb{N}}$ of $\{\L_t(f), t\in [0,T]\}_{f \in \mathbb{R}[x]}$ satisfies
	\begin{align*}
		\L_t(1) = 0, \quad \L_t(x) = - \L_0(x) + G_t(x) - \dfrac{1}{2} e^{-t/2} \int_0^t e^{s/2} \left( G_s(x) - \L_0(x) \right) ds,
	\end{align*}
	and for $n \ge 0$,
	\begin{align} \label{OU limit equation}
		\L_t(x^{n+2})
		&= e^{-\frac{n+2}{2}t} \mathcal{L}_0(x^{n+2})+R_t(n) + G_t(x^{n+2})  \notag \\
		&\quad - \dfrac{n+2}{2} e^{-\frac{n+2}{2} t} \int_0^t e^{\frac{n+2}{2} s} (R_s(n) + G_s(x^{n+2})) ds.
	\end{align}
	where
	\begin{align}\label{eq-rtn}
		R_t(n)
		&= \dfrac{(n+2)(n+1)}{4} \int_{0}^t \langle x^n, \mu_s \rangle ds + \dfrac{n+2}{2} \sum_{k=0}^n \int_0^t \mathcal{L}_s (x^{n-k}) \mu_s(x^k) ds
	\end{align}
	and $\{G_t(x^n), t \in [0,T]\}_{n \in \mathbb{N}}$ is a centered Gaussian family with the covariance
	\begin{align}\label{eq3.28}
		\mathbb{E} \left[ G_t(x^n) G_s(x^m) \right]
		= mn \int_0^{t \wedge s} \langle x^{n+m-2}, \mu_u \rangle du, \quad n,m \ge 1.
	\end{align}
	\end{enumerate}
\end{theorem}

\begin{proof}

Consider the symmetric OU matrix $X^N_t$, of which the entries $\{X_{ij}^N(t)\}$  satisfy 
\begin{align} \label{OU entry SDE}
	d X_{ij}^N(t) = - \dfrac{1}{2} X_{ij}^N(t) dt + \dfrac{2\delta_{ij} + \sqrt{2}(1-\delta_{ij})}{2\sqrt{N}} dB_{ij}(t),~~ 1\le i\le j\le N, ~ t\ge 0,
\end{align}
where $\{B_{ij}(t), i \le j\}$ is a family of independent Brownian motions. Denoting by \[\sigma_{ij} = \dfrac{2\delta_{ij} + \sqrt{2}(1-\delta_{ij})}{2\sqrt{N}},\] the solution to \eqref{OU entry SDE} is given by 
\begin{align*}
	X_{ij}^N(t)
	= X_{ij}^N(0) e^{-t/2} + \sigma_{ij} e^{-t/2} \int_0^t e^{s/2} dB_{ij}(s).
\end{align*}
The stochastic integral is a martingale with quadratic variation
\begin{align*}
	\left\langle \int_0^{\cdot} e^{s/2} dB_{ij}(s) \right\rangle_t
	= e^t - 1.
\end{align*}
By Knight's Theorem, there exists a family of independent standard one-dimensional Brownian motions $\{\tilde{B}_{ij}(t), i \le j\}$, such that
\begin{align*}
	\int_0^t e^{s/2} dB_{ij}(s) = \tilde{B}_{ij}(e^t-1).
\end{align*}
Thus, we have
\begin{align} \label{eq-3.42}
	X_{ij}^N(t)
	= e^{-t/2}\left( X_{ij}^N(0)  + \sigma_{ij}  \tilde{B}_{ij}(e^t-1)\right).
\end{align}

Let $Y^N_t$ be a matrix-valued stochastic process whose entries $\{Y_{ij}^N(t), i \le j\}$ are given by
\begin{align} \label{eq-3.43}
	Y_{ij}^N(t) = Y_{ij}^N(0) + \sqrt{2} \sigma_{ij} \tilde{B}_{ij}(t),
\end{align}
with $Y_{ij}^N(0) = \sqrt{2} X_{ij}^N(0), 1 \le i \le j \le N$. Then $Y^N$ is the scaled symmetric Brownian motion introduced in section \ref{sec-dyson}. By \eqref{eq-3.42} and \eqref{eq-3.43}, \[\sqrt{2} e^{t/2} X_{ij}^N(t) = Y_{ij}^N(e^t-1), ~ 1\le i \le j\le N,\] 
and hence 
\begin{align*}
	\sqrt{2} e^{t/2} \lambda_i^N(t) = \tilde{\lambda}_i^N(e^t-1), ~ 1\le i \le N,
\end{align*}
where $\{\lambda_i^N(t)\}$ and $\{\tilde \lambda_i^N(t)\}$ are the eigenvalues of $X^N(t)$ and $Y^N(t)$, respectively. 

Thus, almost surely, we have
\begin{align*}
	\langle |x|^p, L_N(t) \rangle
	&= \dfrac{1}{N} \sum_{i=1}^N |\lambda_i^N(t)|^p \\
	&= 2^{-p/2} e^{-pt/2} \dfrac{1}{N} \sum_{i=1}^N |\tilde{\lambda}_i^N(e^t-1)|^p \\
	&= 2^{-p/2} e^{-pt/2} \langle |x|^p, \tilde{L}_N(e^t-1) \rangle, ~\forall t>0,
\end{align*}
where $L_N(t)$  and $\tilde{L}_N(t)$ are the empirical measures of $\{\lambda_i^N(t)\}$  and $\{\tilde{\lambda}_i^N(t)\}$, respectively. Note that $\tilde{\lambda}_i^N(0) = \sqrt{2} \lambda_i^N(0)$ satisfies  condition \eqref{eq-bound-initial} in Theorem \ref{CLT for Dyson} with the constants $a$ and $b$ replaced by $2a$ and $\sqrt{2}b$. By the estimation \eqref{eq3.29}, for all $p \ge 1$ and $N \ge \alpha p$ for some positive constant $\alpha$, we have
\begin{align}
	\mathbb{E} \left[ \sup_{t \in [0,T]} \langle |x|^p, L_N(t) \rangle \right]
	&\le 2^{-p/2} \mathbb{E} \left[ \sup_{t \in [0,e^T-1]} \langle |x|^p, \tilde{L}_N(t) \rangle \right]\notag \\
	&\le 2^{-p/2} C(2a,\sqrt{2}b,e^T-1)^p\notag \\
	&= C'(a,b,T)^p,\label{eq-3.45}
\end{align}
where $C'(a,b,T)$ is positive constant depending only on $(a,b,T)$.

Thus,  by  Lemma \ref{largest eigen uniform bound} and  Corollary \ref{Coro-Lp},  $Q_t^N(x^n)$ defined by \eqref{eq-Q} converges in distribution to a centered Gaussian family $\{G_t(x^n), t\in[0,T]\}_{n\in \mathbb N}$ with covariance given by \eqref{eq3.28}. Similar to \eqref{eq-3.16}, for $n \ge -1$, we have
\begin{align*}
	Q_t^N(x^{n+2})
	=& \mathcal{L}_t^N(x^{n+2}) - \mathcal{L}_0^N(x^{n+2}) + \dfrac{n+2}{2} \int_{0}^t \mathcal{L}_s^N(x^{n+2}) ds - \dfrac{(n+2)(n+1)}{4} \int_{0}^t \langle x^n, \mu_s \rangle ds \notag\\
	&- \dfrac{n+2}{2} \sum_{k=0}^n \int_0^t \mathcal{L}_s^N (x^{n-k}) \mu_s(x^k) ds - \dfrac{(n+2)}{4N} \sum_{k=0}^n \int_{0}^t \mathcal{L}_s^N (x^{n-k}) \mathcal{L}_s^N (x^k) ds.
\end{align*}
 Letting $N\to \infty$, we have
\begin{align*}
	G_t(x^{n+2})
	\overset d=& \mathcal{L}_t(x^{n+2}) - \mathcal{L}_0(x^{n+2}) + \dfrac{n+2}{2} \int_{0}^t \mathcal{L}_s(x^{n+2}) ds - \dfrac{(n+2)(n+1)}{4} \int_{0}^t \langle x^n, \mu_s \rangle ds \notag\\
	&- \dfrac{n+2}{2} \sum_{k=0}^n \int_0^t \mathcal{L}_s (x^{n-k}) \mu_s(x^k) ds\\
	=&\mathcal{L}_t(x^{n+2}) - \mathcal{L}_0(x^{n+2}) + \dfrac{n+2}{2} \int_{0}^t \mathcal{L}_s(x^{n+2}) ds-R_t(n)
\end{align*}
where $R_t(n)$ is given in \eqref{eq-rtn}. Without loss of generality, we may replace ``$\overset d =$" by ``$=$" in the above equation. Thus we have 
\begin{align*}
	\mathcal{L}_t(x^{n+2}) + \dfrac{n+2}{2} \int_{0}^t \mathcal{L}_s(x^{n+2}) ds = \L_0(x^{n+2})+G_t(x^{n+2})+R_t(n),
\end{align*}
whose solution is given by \eqref{OU limit equation}. 

The proof is concluded.
\end{proof}

Now we extend the result of Theorem \ref{CLT for OU} to a generalized system of \eqref{OU eigenvalue SDE}.

\begin{corollary} \label{general OU}
Consider the following SDEs
\begin{align} \label{OU SDE with drift}
	d \lambda_i^N(t)
	= \dfrac{1}{\sqrt{N}} dW_i(t) + \left( b_N(\lambda_i^N(t)) + \dfrac{1}{2N} \sum_{j:j \neq i} \dfrac{1}{\lambda_i^N(t) - \lambda_j^N(t)} \right) dt, ~ 1\le i \le N, ~ t\ge 0,
\end{align}
where $b_N(x)$ satisfies, for some constant $c \in \mathbb{R}$,
\begin{align} \label{con-gen-OU-b_N}
	\lim_{N \rightarrow \infty} N \left\| b_N(x) + \dfrac{1}{2}x - c \right\|_{L^{\infty}(\mathbb{R})} = 0.
\end{align}
Furthermore, assume the same initial conditions as in Theorem \ref{CLT for OU}. Then the conclusion of Theorem \ref{CLT for OU} still holds with $R_t(n)$ in \eqref{eq-rtn} replaced by
\begin{align*}
	R_t(n)
	&= c(n+2) \int_0^t \L_s^N(x^{n+1}) ds + \dfrac{(n+2)(n+1)}{4} \int_{0}^t \langle x^n, \mu_s \rangle ds \\
	&+ \dfrac{n+2}{2} \sum_{k=0}^n \int_0^t \mathcal{L}_s (x^{n-k}) \mu_s(x^k) ds.
\end{align*}
\end{corollary}

\begin{proof}
The proof is similar to the proof of Corollary \ref{general Dyson}, which is sketched below.

By \eqref{con-gen-OU-b_N}, without loss of generality, we assume
\begin{align*}
	-\dfrac{1}{2} x + c - 1 \le b_N(x) \le -\dfrac{1}{2} x + c + 1,
\end{align*}
for all $N \ge 1$. Then we have
\begin{align}\label{eq-3.48}
	\mathbb{P} \left( x_i^N(t) \le \lambda_i^N(t) \le y_i^N(t), \ \forall 1 \le i \le N, \ \forall t > 0 \right) = 1,
\end{align}
where  the processes $(x_i^N(t))_{1 \le i \le N}$ and $(y_i^N(t))_{1 \le i \le N}$ are the solutions of the following systems of SDEs respectively:
\begin{align*}
	d x_i^N(t)
	= \dfrac{1}{\sqrt{N}} dW_i(t) + \left( - \dfrac{1}{2} x_i^N(t) + c - 1 + \dfrac{1}{2N} \sum_{j:j \neq i} \dfrac{1}{x_i^N(t) - x_j^N(t)} \right) dt, ~ 1\le i \le N, ~ t\ge 0,
\end{align*}
and
\begin{align*}
	d y_i^N(t)
	= \dfrac{1}{\sqrt{N}} dW_i(t) + \left( - \dfrac{1}{2} y_i^N(t) + c + 1 + \dfrac{1}{2N} \sum_{j:j \neq i} \dfrac{1}{y_i^N(t) - y_j^N(t)} \right) dt, ~ 1\le i \le N, ~ t\ge 0,
\end{align*}
with the initial conditions $x_i^N(0) = y_i^N(0) = \lambda_i^N(0)$ for $1 \le i \le N$. Noting that $(x_i^N(t)-2c+2)_{1 \le i \le N}$ and $(y_i^N(t)-2c-2)_{1 \le i \le N}$ solve the SDEs \eqref{OU eigenvalue SDE}, by \eqref{eq-3.45} and \eqref{eq-3.48}, we get that the uniform $L^p$ bound \eqref{eq-p-bound} holds for system \eqref{OU SDE with drift}.

Then applying Corollary \ref{Coro-Lp} and following the approach in the proof of Theorem \ref{CLT for Wishart}, we get the desired result.
\end{proof}

\section{Useful lemmas}\label{sec-appendix}

In this section, we provide some results that were used in the preceding sections.

The following CLT for martingales was used in the proof of Theorem \ref{Thm-Wishart}.

\begin{lemma}[Rebolledo's Theorem] \label{Rebolledo Thm}
Let $n \in \mathbb{N}$, and let $\{M_N\}_{N \in \mathbb{N}}$ be a sequence of continuous centered martingales with values in $\mathbb{R}^n$. If the quadratic variation $\langle M_N \rangle_t$ converges in $L^1(\Omega)$ to a continuous deterministic function $\phi(t)$ for all $t>0$, then for any $T > 0$, as a continuous process from $[0,T]$ to $\mathbb{R}^n$, $(M_N(t), t \in [0,T])$ converges in law to a Gaussian process $G$ with mean 0 and covariance
\begin{align*}
	\mathbb{E} [G(s) G(t)^\intercal] = \phi(t \wedge s).
\end{align*}
\end{lemma}

Section \ref{section-comparison} was based on the following comparison principle for multi-dimensional SDEs which is a direct consequence of \cite[Theorem 1.1 and Theorem 1.2]{Geiss1994}.

\begin{lemma} \label{comparison}
On a certain complete probability space equipped with a filtration that satisfies the usual conditions (\cite[Definition 2.25]{Karatzas1998}), consider the following SDEs
\begin{align} \label{Compare-sde}
	\begin{aligned}
	Y(t) = Y(0) + \int_0^t b^{(1)}(s,Y(s)) ds + \int_0^t \sigma(s,Y(s)) dW(s),\\
	Z(t) = Z(0) + \int_0^t b^{(2)}(s,Z(s)) ds + \int_0^t \sigma(s,Z(s)) dW(s),
	\end{aligned}
\end{align}
where $\{W(t), t\ge 0\}$ is a $d$-dimensional Brownian motion.  Assume the solutions to  SDEs \eqref{Compare-sde} are  pathwisely unique and non-exploding. If the following conditions are satisfied, 
\begin{enumerate}
	\item the drift functions $b^{(1)}(t,x)$ and $b^{(2)}(t,x)$ are continuous mappings from $[0,\infty) \times \mathbb{R}^n$ to $\mathbb{R}^n$. Besides, they are quasi-monotonously increasing in the sense that for $1 \le i \le n$ and $j = 1,2$, $b_i^{(j)}(t,x) \le b_i^{(j)}(t,y)$, whenever $x_i = y_i$ and $x_l \le y_l$ for $l \in \{1, \ldots, n \} \setminus \{i\}$;
	\item the dispersion matrix $\sigma(t,x)$ is a continuous mapping from $[0,\infty) \times \mathbb{R}^n$ to $\mathbb{R}^{n \times d}$ that satisfies the following condition
	\begin{align*}
		\sum_{j=1}^d |\sigma_{ij}(t,x) - \sigma_{ij}(t,y)| \le \rho(|x_i - y_i|)
	\end{align*}
	for all $t \ge 0$ and $x = (x_1, \ldots, x_n)^\intercal, y = (y_1, \ldots, y_n)^\intercal \in \mathbb{R}^n$, where $\rho: [0,\infty) \rightarrow [0,\infty)$ is a strictly increasing function with $\rho(0) = 0$ and
	\begin{align*}
		\int_{0+} \rho^{-2}(u) du = \infty;
	\end{align*}
	\item $b_i^{(1)}(t,x) \le b_i^{(2)}(t,x)$ for all $1 \le i \le n$, $t \ge 0$, $x \in \mathbb{R}^n$;
	\item for $1 \le i \le n$, $Y_i(0) \le Z_i(0)$ almost surely,
\end{enumerate}
then we have
\begin{align*}
	\mathbb{P} \left( Y_i(t) \le Z_i(t), \forall t \ge 0, 1 \le i \le n \right) = 1.
\end{align*}
\end{lemma}
The following lemma was employed in the proof of Proposition \ref{largest eigenvalue bound}.

\begin{lemma} \label{Lemma-stationary distribution}
Let $u^N(t)$ be the strong solution to \eqref{self-similarity sde}. If $u^N(t)$ is distributed according to $\mathbb{P}^N$ in \eqref{density of Wishart eigenvalue}, then for $f \in C_b^2(\mathbb{R}^N)$,
\begin{align*}
	\dfrac{d}{dt} \mathbb{E} [f(u^N(t))] = 0.
\end{align*}
\end{lemma}

\begin{proof}
For $f \in C_b^2(\mathbb{R}^N)$, applying It\^o's formula to \eqref{self-similarity sde}, we have
\begin{align*}
	f(u^N(t))
	&= f(u^N(t_0)) + \sum_{i=1}^N \int_{0}^t \partial_i f(u^N(s)) \cdot 2 \dfrac{\sqrt{u_i^N(s)}}{\sqrt{N(s+a)}} dW_i(s) \\
	&\quad + \sum_{i=1}^N \int_{0}^t \partial_i f(u^N(s)) \cdot \dfrac{1}{s+a} \left( \dfrac{P}{N} - u_i^N(s) + \dfrac{1}{N} \sum_{j:j \neq i} \dfrac{u_i^N(s) + u_j^N(s)}{u_i^N(s) - u_j^N(s)} \right) ds \\
	&\quad + \dfrac{1}{2} \sum_{i=1}^N \int_{0}^t \partial_i^2 f(u^N(s)) \cdot 4 \dfrac{u_i^N(s)}{N(s+a)} ds.
\end{align*}
Here, $\partial_i$ is the partial derivative with respect to the $i$-th component $x_i$. Therefore, for $t\ge 0$,
\begin{align*}
	\frac{d}{dt} \mathbb{E} \left[ f(u^N(t)) \right]
	&= \mathbb{E} \left[ \dfrac{1}{t+a} \sum_{i=1}^N \partial_i f(u^N(t)) \cdot \left( \dfrac{P}{N} - u_i^N(t) \right) \right] \\
	&\quad + \mathbb{E} \left[ \dfrac{1}{N(t+a)} \sum_{i \neq j} \partial_i f(u^N(t)) \cdot \dfrac{u_i^N(t) + u_j^N(t)}{u_i^N(t) - u_j^N(t)} \right] \\
	&\quad + \mathbb{E} \left[ \dfrac{2}{N(t+a)} \sum_{i=1}^N \partial_i^2 f(u^N(t)) u_i^N(t) \right].
\end{align*}
Thus it suffices to show, with the density function $p(x)$ in \eqref{density of Wishart eigenvalue},
\begin{align} \label{stationary}
	&\sum_{i=1}^N \int_{\Delta_N} \partial_i f(x) \cdot \left( \dfrac{P}{N} - x_i \right) p(x) dx
	+ \dfrac{1}{N} \sum_{i \neq j} \int_{\Delta_N} \partial_i f(x) \cdot \dfrac{x_i + x_j}{x_i - x_j} p(x) dx \nonumber \\
	&+ \dfrac{2}{N} \sum_{i=1}^N \int_{\Delta_N} \partial_i^2 f(x) x_i p(x) dx = 0,
\end{align}
where $\Delta_N = \{x \in \mathbb{R}^N: 0 < x_1 < \ldots < x_N \}$ is the support of $\mathbb{P}^N$. Noting that $p(x)$ vanishes on $\partial \Delta_N$, we have by the integration by parts formula, 
\begin{align*}
	\int_{\Delta_N} \partial_i^2 f(x) x_i p(x) dx
	&= \int_{\partial \Delta_N} \partial_i f(x) x_i p(x) dS - \int_{\Delta_N} \partial_i f(x) \partial_i \left( x_i p(x) \right) dx \\
	&= - \int_{\Delta_N} \partial_i f(x) \left( p(x) + x_i \partial_i p(x) \right) dx.
\end{align*}
Hence, to show \eqref{stationary}, it is sufficient to verify
\begin{align*}
	\sum_{i=1}^N \left( \dfrac{P}{N} - x_i \right) p(x) + \dfrac{1}{N} \sum_{i \neq j} \dfrac{x_i + x_j}{x_i - x_j} p(x) - \dfrac{2}{N} \sum_{i=1}^N \left( p(x) + x_i \partial_i p(x) \right) = 0.
\end{align*}
By the chain rule,
\begin{align*}
	\partial_i p(x) = - \dfrac{N}{2} p(x) + \dfrac{P-N-1}{2} \dfrac{1}{x_i} p(x) + \sum_{j:j \neq i} \dfrac{1}{x_i - x_j} p(x).
\end{align*}
Hence,
\begin{align*}
	\dfrac{2}{N} \sum_{i=1}^N x_i \partial_i p(x)
	&= - \sum_{i=1}^N x_i p(x) + (P-N-1) p(x) + \dfrac{2}{N} \sum_{i \neq j} \dfrac{x_i}{x_i - x_j} p(x) \\
	&= - \sum_{i=1}^N x_i p(x) + (P-N-1) p(x) + \dfrac{1}{N} \sum_{i \neq j} \left( \dfrac{x_i + x_j}{x_i - x_j} + 1 \right) p(x) \\
	&= - \sum_{i=1}^N x_i p(x) + (P-2) p(x) + \dfrac{1}{N} \sum_{i \neq j} \dfrac{x_i + x_j}{x_i - x_j} p(x),
\end{align*}
which gives the desired result.
\end{proof}

\bibliography{WishartProcess}

\end{document}